\documentclass[11pt,reqno]{amsart}
\usepackage{amssymb,amsmath,amsthm}
\usepackage{graphicx}
\usepackage{subfigure}
\usepackage{color}
\usepackage{array}
\usepackage{ulem}
\newtheorem{theorem}{Theorem}[section]
\newtheorem{lemma}{Lemma}[section]
\newtheorem{proposition}{Proposition}[section]
\newtheorem{corollary}{Corollary}[section]
\newtheorem{remark}{Remark}[section]
\newtheorem{conjecture}{Conjecture}[section]

\theoremstyle{definition}

\theoremstyle{remark}

\makeatletter
\@namedef{subjclassname@2020}{2020 Mathematics Subject Classification}
\makeatother

\numberwithin{equation}{section}

\date{\today}
\thanks{ Email: \dag schen@ccnu.edu.cn. \ddag duan@gbu.edu.cn. }

\begin{document}
%\today

\title[$C^{k}$ extension and invariant manifolds]
{$C^{k}$ extension and invariant manifolds for the compactification of nonautonomous \\ systems with autonomous limits}

\maketitle

%\author{Shuang Chen}
%\address{School of Mathematics and Statistics}
%\curraddr{ }
%\email{schen@hust.edu.cn}
\centerline{\scshape Shuang Chen\,$^{a,b,\dag}$, \  Jinqiao Duan\,$^{c,\ddag}$}
\medskip
{\footnotesize
% please put the address of the first author
 \centerline{$^{a}$ School of Mathematics and Statistics, Central China Normal University}
 \centerline{Wuhan, Hubei 430079, China}
 \centerline{$^{b}$ Key Lab NAA--MOE, Central China Normal University,}
 \centerline{Wuhan, Hubei 430079, China}
%centerline{P.O. Box 71010, Wuhan 430079, P. R. China}
   \centerline{$^{c}$ Department of Mathematics, School of Sciences, Great Bay University}
   \centerline{Dongguan, Guangdong 523000, China}
} % Do not forget to end the {\footnotesize by the sign }

\medskip

\begin{abstract}
We study the compactification of nonautonomous systems with autonomous limits and related dynamics.
Although the $C^{1}$ extension of the compactification was well established,
a great number of problems arising in bifurcation and stability analysis
require the compactified systems with high-order smoothness.
Inspired by this, we give a criterion for the $C^{k}$ ($k\geq 2$) extension of the compactification.
After compactifying nonautonomous systems,
the compactified systems may gain an additional center direction.
We prove the existence and uniqueness of center or center-stable manifolds for general compact invariant sets
including normally hyperbolic invariant manifolds and admissible sets.

\end{abstract}

\parskip 0.3cm

\section{Introduction}

Nonautonomous dynamical systems naturally arise in modeling real world systems and theoretical studies,
such as mechanical systems with external forces and stability analysis of solutions.
Compared to autonomous dynamical systems,
whose evolutions are independent of initial time,
nonautonomous dynamical systems have time-dependent terms explicitly.
This causes a big obstacle in analyzing the evolution of a nonautonomous dynamical system.
In the past few decades,
substantial progress has been made in the study of various aspects of nonautonomous dynamical systems.
See \cite{Carvalho-etal-13,Coppel-78,Jager-09,Johnson-etal-16,Kloeden-Rasmussen-11,Li-Duan-09,Lu-Wang-10,Sun-Cao-Duan-07} for instance.
In the present paper,
we focus on  nonautonomous differential systems with autonomous limits,
especially for their compactification framework,
and the $C^{k}$ extension and invariant manifolds of the related compactified systems.

More precisely, consider a nonautonomous differential system
\begin{eqnarray}\label{eq:NDS}
\dot x:=\frac{d x}{dt}=f(x,\mu(t)),
\end{eqnarray}
where $x\in U$ is the state variable, $t$ is time, and $U$ is a open subset of $\mathbb{R}^{N}$.
Here the time-dependent term $\mu:\mathbb{R}\to V\subset \mathbb{R}^{d}$
and the vector field $f: U\times V\to \mathbb{R}^{N}$
are $C^{k}$ smooth for some integer $k\geq 2$.
We focus on the time-dependent term $\mu$ with asymptotic limits, i.e.,
\begin{eqnarray*}
\lim_{t\to +\infty}\mu(t)=\mu^{+}
\ \ \  \mbox{ or }\ \ \
\lim_{t\to -\infty}\mu(t)=\mu^{-}.
\end{eqnarray*}
If both of these two limits hold,
then the term $\mu$ is said to be  bi-asymptotically constant with future limit $\mu^{+}$ and past limit $\mu^{-}$.
If one of them holds,
then the term $\mu$ is said to be asymptotically constant with future limit $\mu^{+}$ or past limit $\mu^{-}$.
The nonautonomous system \eqref{eq:NDS} with asymptotically constant limits appears in  a wide range of problems,
such as in bifurcation analysis of homoclinic and heteroclinic orbits \cite{Guck-Holmes-83,Ilyashenko-Li-99,Palmer-84,Zhu-Zhang-20},
stability analysis of nonlinear waves and patterns \cite{Alexander-Gardner-Jones-90,Kapitula-Promislow-13,Sandstede-02},
studying forced waves \cite{Berestycki-Fang-18} and rate-induced critical transitions \cite{Ashwin-etal-12,Kuehn-Longo-22,Wieczorek-Xie-Ashwin-23}, etc.

If $\mu$ has a future limit $\mu^{+}$,
then letting $t\to +\infty$ in the right side of \eqref{eq:NDS} yields
\begin{eqnarray}\label{eq:NDS-future}
\dot x=f(x,\mu^{+}),
\end{eqnarray}
which is called an autonomous future limit system.
If $\mu(t)$ has a past limit, then letting $t\to -\infty$ yields
\begin{eqnarray}\label{eq:NDS-past}
\dot x=f(x,\mu^{-}),
\end{eqnarray}
which is called an autonomous past limit system.
As a result, if $\mu$ is asymptotically constant,
then the nonautonomous system \eqref{eq:NDS} has an autonomous limit \cite{Markus-56}.
By letting $\nu=t$,
the nonautonomous system \eqref{eq:NDS} is changed into an autonomous system
\begin{eqnarray}\label{ode-auto}
\begin{aligned}
\dot x &=f(x,\mu(\nu)),\\
\dot \nu &=1,
\end{aligned}
\end{eqnarray}
where $\nu$ is defined on $(-\infty,\infty)$.
As far as we know,
the compactification approaches, such as the Poincar\'{e} compactification and the Poincar\'{e}--Lyapunov Compactification
\cite{Dumortier-Llibre-Artes-06,ZTHD-92},
are good for analyzing the behavior of orbits near infinity.
It is natural to compactify $\nu$-direction so that we are able to grasp the limit behavior of the nonautonomous system \eqref{eq:NDS}.
Wieczorek, Xie and Jones \cite{Wieczorek-Xie-Jone-21} recently developed a framework
for compactifying $\nu$-direction as follows.
Let $\phi: \mathbb{R}\to (-1,1)$ be a function which is at least $C^{2}$-smooth  and satisfies
\begin{eqnarray*}
\lim_{t\to \pm \infty}\phi(t)=\pm 1, \ \ \
\dot \phi(t)>0 \mbox{ for } t\in\mathbb{R},\ \ \
\lim_{t\to \pm \infty}\dot \phi(t)=0.
\end{eqnarray*}
Then by letting $\nu=\phi^{-1}(s):=h(s)$, system \eqref{ode-auto} is changed into
\begin{eqnarray}\label{ode-compact}
\begin{aligned}
\dot x &=f(x,\mu(h(s))),\\
\dot s &=g(s),
\end{aligned}
\end{eqnarray}
where $h(s)=\phi^{-1}(s)=t$ and $g(s)=\dot\phi(h(s))$.
Clearly, the additional dependent variable $s$ is defined on $(-1,1)$
and $\nu=\pm \infty$ are transformed into $s=\pm 1$.
This realizes the compactification of $\nu$-direction.
The process is referred to the compactification of the nonautonomous system \eqref{eq:NDS} with autonomous limits.

However, the compactified system \eqref{ode-compact} may not be $C^{1}$-smooth at the added invariant subspaces $\{s=\pm 1\}$,
as pointed out in section 2.1.2 of \cite{Wieczorek-Xie-Jone-21}.
In \cite{Wieczorek-Xie-Jone-21}, Wieczorek, Xie and Jones  derived  a criterion
such that the $C^{1}$ smoothness of system \eqref{ode-compact} can be extended to $U\times [-1,1]$.
This is called the  $C^{1}$ extension of the compactification.
The framework presented in \cite{Wieczorek-Xie-Jone-21} provides a possibility to investigate the dynamics of the nonautonomous system \eqref{eq:NDS}
by the limit systems \eqref{eq:NDS-future} and \eqref{eq:NDS-past} in the view of perturbation theory,
while the $C^{1}$ smoothness of autonomous differential systems is not enough to analyze stability and bifurcations
in some degenerate cases.
For example,
the linearization of a planar system about a higher-order weak focus has a pair of purely imaginary eigenvalues,
and the corresponding first Lyapunov coefficient has zero real part (see, for instance, \cite{Chow-Hale-82,Hassard-81}).
In general, we require high-order smoothness conditions to give the explicit expressions of the high-order Lyapunov coefficients.
Then we can analyze its stability and the corresponding degenerate Hopf bifurcation.
Inspired by it,
one of our contributions is to give a criterion under which
the compactified system is $C^{k}$-smooth for any $k\geq 2$ on the extended domain $U\times [-1,1]$,
i.e., realizing the $C^{k}$ extension of the compactification.

After compactifying the nonautonomous system \eqref{eq:NDS},
one of the remaining problems is to study the invariant manifolds of compact invariant sets
for the limit systems when embedded in the extended phase space.
This is helpful to grasp the dynamics of the nonautonomous system \eqref{eq:NDS},
such as, exploring its local pullback attractors, connecting heteroclinic orbits, etc.
Unluckily, the compactified system may gain an additional center direction (see the detail in subsection \ref{sec:equiliria}).
Due to the important role of the uniqueness of center manifolds in applications,
Wieczorek, Xie and Jones \cite{Wieczorek-Xie-Jone-21}
further established the existence and uniqueness of center-stable manifolds for
the hyperbolic saddles
of the future limit system \eqref{eq:NDS-future}
when embedded in the extended phase space (see \cite[Theorems 3.2]{Wieczorek-Xie-Jone-21}),
and center manifolds for the hyperbolic sink of the past limit system \eqref{eq:NDS-past}
when embedded in the extended phase space (see \cite[Theorem 3.4]{Wieczorek-Xie-Jone-21}).
It is worth noting that
the argument for the uniqueness in \cite{Wieczorek-Xie-Jone-21} relies on the good local structures of equilibria
so that we are able to  normalize the compactified system near equilibria.
However,  many invariant sets have no good local structures, even without tubular neighborhoods
\cite{Arnold-Afrajmovich-Ilyashenko-13,Chow-Liu-Yi-00,Robinson-99}.
Then the method used in \cite{Wieczorek-Xie-Jone-21} is not applicable at all.

In this paper,
we consider more general compact invariant sets
including normally hyperbolic invariant manifolds \cite{Bates-Peter-Lu-98,Fenichel-71,Wiggins-94} and admissible sets \cite{Chow-Liu-Yi-00},
and  partially proves a conjecture proposed by Wieczorek, Xie and Jones \cite{Wieczorek-Xie-Jone-21} (see Conjecture \ref{conj:WXJ} below).
If the invariant set is a normally hyperbolic invariant manifold,
we prove the existence of its center-stable or center manifolds
when embedded in the extended phase space, following the idea of Fenichel \cite{Fenichel-71,Fenichel-79}.
By constructing a suitable tubular neighborhood,
we prove the uniqueness with the help of some explicit estimates along the orbits in this tubular neighborhood.
If it is an admissible set, which may have no tubular neighborhoods,
we prove the existence by  several results in Chow, Liu and Yi's work \cite{Chow-Liu-Yi-00}
and the uniqueness by a similar argument used in the case of normally hyperbolic invariant manifolds.
Our developed theory ensures that the solutions of interest lie in
the unique invariant manifolds of invariant sets for the limit systems
when embedded in the extended phase space.
This will be useful in various problems arising from applications 
\cite{Berestycki-Fang-18,Kloeden-Rasmussen-11,Kuehn-Longo-22,Wieczorek-Xie-Jone-21,Wieczorek-Xie-Ashwin-23}.

This paper is organized as follows.
In section \ref{sec:C-k},
we provide the compactification framework and give the criteria for the $C^{k}$ extensions of
the two-sided compactification if $\mu$ is  bi-asymptotically constant
and the one-sided compactification if $\mu$ is asymptotically constant.
We review the previous results of Wieczorek, Xie and Jones and introduce their conjecture in subsection \ref{sec:equiliria}.
Then we prove Wieczorek, Xie and Jones's conjecture holds for normally hyperbolic invariant manifolds in subsection \ref{sec:NHIM},
and normally hyperbolic admissible sets in subsection \ref{sec:Adm}.
We also apply our results to a mechanical model in the final section.

\section{$C^{k}$-smooth compactified systems}
\label{sec:C-k}

In this section,
we consider the $C^{k}$ extension for the compactification of the nonautonomous system \eqref{eq:NDS}
with asymptotically autonomous properties.
According to the types of asymptotical properties,
the discussion is divided into two different cases:
{\bf (C1)} $\mu$ is bi-asymptotically constant with future limit $\mu^{+}$ and past limit $\mu^{-}$,
and {\bf (C2)} $\mu$ is asymptotically constant with future limit $\mu^{+}$ or past limit $\mu^{-}$.
The $C^{1}$ extension of the related compactification was previously studied
by Wieczorek, Xie and Jones in \cite{Wieczorek-Xie-Jone-21}.
We focus on the explicit conditions such that the compactified systems admit higher regularity.

\subsection{Two-sided compactification}
\label{sec:two-side}

In this subsection,
we consider the case that the time-dependent term $\mu$ is bi-asymptotically constant with future limit $\mu^{+}$ and past limit $\mu^{-}$.
In order to realize the two-side compactification,
we take a change $s=\phi(t)$ and require that $\phi$ satisfies the assumption:
\begin{enumerate}
\item[{\bf (H1)}] The function $\phi:\mathbb{R}\to (-1,1)$ is $C^{k+1}$-smooth and satisfies
\begin{eqnarray*}
\lim_{t\to \pm \infty}\phi(t)=\pm 1, \ \ \
\dot \phi(t)>0 \mbox{ for } t\in\mathbb{R},\ \ \ \lim_{t\to \pm \infty}\dot \phi(t)=0.
\end{eqnarray*}
\end{enumerate}
After this change, the nonautonomous system \eqref{eq:NDS} is transformed into
an autonomous system:
\begin{eqnarray}\label{eq:transf}
\begin{aligned}
\dot x &=f(x,\mu(h(s))),\\
\dot s &=g(s),
\end{aligned}
\end{eqnarray}
where
\begin{eqnarray*}
h(s)=\phi^{-1}(s)=t,\ \ \ \ \ g(s)=\phi'(h(s))=1/h'(s).
\end{eqnarray*}
Here we shall always use $\cdot$ and $'$ to denote $d/dt$ and $d/ds$, respectively.

It is clear that the transformed system \eqref{eq:transf} is well defined on the set $U\times (-1,1)$.
By assumption {\bf (H1)},
the transformation $\phi$ is $C^{k+1}$-smooth and $\dot \phi(t)>0$ for all $t\in\mathbb{R}$.
Consider equation $\phi(t)=s$. Then by the implicit function theorem \cite[Theorem 2.3, p.26]{Chow-Hale-82},
we have that $h(s)=\phi^{-1}(s)$ is $C^{k+1}$-smooth for $s\in(-1,1)$.
This together with the $C^{k}$ smoothness of $f$ and $g$ implies
that the transformed system \eqref{eq:transf} is $C^{k}$-smooth on the set $U\times (-1,1)$.

Note that the time-dependent term $\mu$ is bi-asymptotically constant with future limit $\mu^{+}$ and past limit $\mu^{-}$.
Then the compactification can be realized by extending $f$ and $g$ to $U\times [-1,1]$ as follows:
\begin{eqnarray*}
F(x,s)\!\!\!&=&\!\!\!\left\{
\begin{aligned}
&f(x,\mu^{-}), && \mbox{ if }\ s=-1,\\
&f(x,\mu(h(s))), && \mbox{ if }\ -1<s<1,\\
&f(x,\mu^{+}), && \mbox{ if }\ s=1,
\end{aligned}
\right.
\\
G(s)\!\!\!&=&\!\!\!\left\{
\begin{aligned}
& 0, && \mbox{ if }\ s=-1,\\
& g(s), && \mbox{ if }\ -1<s<1,\\
& 0, && \mbox{ if }\ s=1.
\end{aligned}
\right.
\end{eqnarray*}
Consequently, we get an autonomous compactified system
\begin{eqnarray}\label{eq:comp}
\begin{aligned}
\dot x &=F(x,s),\\
\dot s &=G(s),
\end{aligned}
\end{eqnarray}
which is defined on $U\times [-1,1]$.

\subsubsection{Transformation conditions for the $C^{k}$ extension}

Now we present some criteria such that
the compactified system \eqref{eq:comp} is also $C^{k}$-smooth on the extended domain $U\times [-1,1]$.
Then we realize the $C^{k}$ extension.
This is not always true, even for $C^{1}$-smoothness (see section 2.1.2 of \cite{Wieczorek-Xie-Jone-21}).
The argument is based on the Fa\`a di Bruno formula  \cite{Faa-55} (see also Appendix A or Formula (1.1) in \cite{Constantine-Savits-96})
for computing arbitrary derivatives of a composition of functions.

For convenience, we introduce some notations as a preparation.
Let $\varphi(\theta)$ be defined on an pen interval $J\subset\mathbb{R}$ and $C^{m}$-smooth on this set.
For each pair of integers $l$ and $n$ with $1\leq l\leq n\leq m$,
we define $S_{n}^{l}\varphi(\theta)$ by
\begin{eqnarray}\label{df:S-n-l}
S_{n}^{l}\varphi(\theta):=\sum_{q(n,l)} n!\prod_{i=1}^{n}\frac{(\varphi^{(i)}(\theta))^{\lambda_{i}}}{(\lambda_{i}!)(i!)^{\lambda_{i}}},
\ \ \ \ \ \theta\in J,
\end{eqnarray}
where
\[
\varphi^{(i)}(\theta):=\frac{d^{i}}{d\theta^{i}}\varphi(\theta),
\ \ \ \ \
q(n,l):=\left\{(\lambda_{1},...,\lambda_{n}): \lambda_{i}\in \mathcal{N}_{0},\
       \sum_{i=1}^{n}\lambda_{i}=l,\ \sum_{i=1}^{n}i\lambda_{i}=n \right\},
\]
and $\mathcal{N}_{0}$ denotes the set of nonnegative integers.

In order to give the criteria for the $C^{k}$-smoothness of $G(s)$,
we define a family of matrix functions $\mathcal{M}_{n}h$ ($0\leq n\leq k$) by
\begin{eqnarray*}
\mathcal{M}_{n}h(s):=
\left(
\begin{array}{cccc}
S_{1}^{1}h(s) & 0 & \cdot\cdot\cdot & 0 \\
S_{2}^{1}h(s) & S_{2}^{2}h(s) & \cdot\cdot\cdot & 0\\
\cdot\cdot\cdot & \cdot\cdot\cdot & \cdot\cdot\cdot & \cdot\cdot\cdot\\
S_{n+1}^{1}h(s) & S_{n+1}^{2}h(s) & \cdot\cdot\cdot & S_{n+1}^{n+1}h(s)
\end{array}
\right),
\ \ \ \ \ s\in (-1,1),
\end{eqnarray*}
where each entry $S_{j}^{i}h(s)$ of $\mathcal{M}_{n}h(s)$ is defined by the formula \eqref{df:S-n-l}.
Let the $i$-th column of $\mathcal{M}_{n}h(s)$ be replaced by the vector $(1,0,...,0)^{T}\in\mathbb{R}^{n+1}$.
Then we get
\begin{eqnarray*}
\mathcal{M}_{n,i}h(s):=\left(
\begin{array}{cccccc}
S_{1}^{1}h(s) & 0 & \cdot\cdot\cdot &  1 & \cdot\cdot\cdot & 0 \\
S_{2}^{1}h(s) & S_{2}^{2}h(s) & \cdot\cdot\cdot & 0 & \cdot\cdot\cdot & 0 \\
\cdot\cdot\cdot & \cdot\cdot\cdot & \cdot\cdot\cdot & \cdot\cdot\cdot& \cdot\cdot\cdot & \cdot\cdot\cdot\\
S_{n+1}^{1}h(s) & S_{n+1}^{2}h(s) & \cdot\cdot\cdot& 0& \cdot\cdot\cdot & S_{n+1}^{n+1}h(s)
\end{array}
\right),\ \ \ \ \ i=1,...,n+1.
\end{eqnarray*}
We shall use the matrix functions $\mathcal{M}_{n}h(s)$ and $\mathcal{M}_{n,i}h(s)$
to give a criterion for the $C^{k}$-smoothness of $G(s)$ in terms of a function of $s$.

Similarly, we define  $\mathcal{Q}_{n}\phi(t)$ ($1\leq n\leq k$) by
\begin{eqnarray}\label{matrix:Q-n}
\mathcal{Q}_{n}\phi(t):=
\left(
\begin{array}{ccccc}
S_{1}^{1}\phi(t) & 0   & \cdot\cdot\cdot & 0 \\
S_{2}^{1}\phi(t) & S_{2}^{2}\phi(t)   & \cdot\cdot\cdot & 0 \\
\cdot\cdot\cdot & \cdot\cdot\cdot  & \cdot\cdot\cdot & \cdot\cdot\cdot\\
S_{n}^{1}\phi(t) & S_{n}^{2}\phi(t)  & \cdot\cdot\cdot & S_{n}^{n}\phi(t)
\end{array}
\right),
\ \ \ \ \ t\in\mathbb{R}.
\end{eqnarray}
Let the $i$-th column of $\mathcal{Q}_{n}\phi(t)$ be replaced by the vector
$(\phi^{(2)}(t),\phi^{(3)}(t),...,\phi^{(n+1)}(t))^{T}$.
Then we obtain
\begin{eqnarray}\label{matrix:Q-n-i}
\mathcal{Q}_{n,i}\phi(t):=\left(
\begin{array}{cccccc}
S_{1}^{1}\phi(t) & 0 & \cdot\cdot\cdot &  \phi^{(2)}(t) & \cdot\cdot\cdot & 0 \\
S_{2}^{1}\phi(t) & S_{2}^{2}\phi(t) & \cdot\cdot\cdot & \phi^{(3)}(t) & \cdot\cdot\cdot & 0 \\
\cdot\cdot\cdot & \cdot\cdot\cdot & \cdot\cdot\cdot & \cdot\cdot\cdot& \cdot\cdot\cdot & \cdot\cdot\cdot\\
S_{n}^{1}\phi(t) & S_{n}^{2}\phi(t) & \cdot\cdot\cdot& \phi^{(n+1)}(t) & \cdot\cdot\cdot & S_{n}^{n}\phi(t)
\end{array}
\right),\ \ \ \ \ i=1,...,n.
\end{eqnarray}
We shall also give another criterion for the $C^{k}$-smoothness of $G(s)$ in terms of a function of $t$
by the matrix functions $\mathcal{Q}_{n}\phi(t)$ and $\mathcal{Q}_{n,i}\phi(t)$.

\begin{lemma}\label{lm:G-smooth}
Suppose that the transformation $\phi$ satisfies assumption {\bf (H1)}.
Let $h(s)=\phi^{-1}(s)$ for $s\in (-1,1)$.
Then $G(s)$ is $C^{k}$-smooth on the interval $[-1,1]$ if the limits
\begin{eqnarray}\label{limt-G-cond}
\lim_{s\to \pm 1^{\mp}}\frac{1}{{\rm det}(\mathcal{M}_{n}h(s))}\sum_{i=1}^{n}{\rm det} (\mathcal{M}_{n,i+1}h(s)) \cdot S_{n}^{i}h(s)
=
\lim_{t\to \pm \infty}\frac{{\rm det}(\mathcal{Q}_{n,n}\phi(t))}{{\rm det} (\mathcal{Q}_{n}\phi(t))}
\end{eqnarray}
exist for all $n=1,...,k$.
\end{lemma}
\begin{proof}
We begin by proving that
\begin{eqnarray*}
{\rm det}(\mathcal{M}_{n}h(s))\neq 0 \ \ \ \ \ \mbox { for }\ s\in (-1,1)\  \mbox{ and }\ 0\leq n\leq k.
\end{eqnarray*}
By the definitions of $q(n,l)$ below \eqref{df:S-n-l}, we see that
\[
q(n,n)=\{(n,0,...,0)\}, \ \ \ \ \ 1\leq n\leq k+1.
\]
Then by \eqref{df:S-n-l}, we can compute
\begin{eqnarray*}
S_{n}^{n}h(s)=(h'(s))^{n}
\end{eqnarray*}
for all $s\in (-1,1)$ and  $1\leq n\leq k+1$.
This together with $h'(s)=1/\dot \phi(t)>0$ yields that
\begin{eqnarray*}
{\rm det}(\mathcal{M}_{n}h(s))=\prod_{i=1}^{n+1}(h'(s))^{i}>0
\ \ \ \mbox { for }  s\in (-1,1).
\end{eqnarray*}
Similarly, we can compute
\begin{eqnarray*}
{\rm det}(\mathcal{Q}_{n}\phi(t))=\prod_{i=1}^{n}(\dot \phi(t))^{i}>0
\ \ \ \mbox { for }\ t\in\mathbb{R}.
\end{eqnarray*}
Hence both $1/{\rm det}(\mathcal{M}_{n}h(s))$ and $1/{\rm det}(\mathcal{Q}_{n}\phi(t))$ are well-defined in this lemma.

By  assumption {\bf (H1)}, we have
\begin{eqnarray*}
\lim_{s\to \pm 1^{\mp}}G(s)=\lim_{s\to \pm 1^{\mp}}g(s)
   =\lim_{t\to \pm \infty}\dot \phi(t)=0.
\end{eqnarray*}
Then $G(s)$ is continuous at $s=\pm 1$.
Thus $G(s)$ is $C^{k}$-smooth on the interval $[-1,1]$ if  the limits
\[
\lim_{s\to \pm 1^{\mp}}G^{(n)}(s)=\lim_{s\to \pm 1^{\mp}}g^{(n)}(s)
\]
exist for all $n=1,2,...,k$.
For this reason, we now give two explicit expressions of $g^{(n)}(s)$ for $s\in (-1,1)$.
For each $n=1,2,...,k$ and $s\in (-1,1)$,
by the Fa\`a di Bruno formula \eqref{FdB-formula} we have
\begin{eqnarray*}
g^{(n)}(s)
 =\sum_{l=1}^{n}\frac{d^{l+1}}{dt^{l+1}}\phi(h(s))\cdot \sum_{q(n,l)} n!\prod_{i=1}^{n}\frac{(h^{(i)}(s))^{\lambda_{i}}}{(\lambda_{i}!)(i!)^{\lambda_{i}}}.
\end{eqnarray*}
By using \eqref{df:S-n-l}, we write $g^{(n)}(s)$ into a compact form
\begin{eqnarray}\label{formula:g-n}
g^{(n)}(s)=\sum_{i=1}^{n}\phi^{(i+1)}(t)\cdot S_{n}^{i}h(s), \ \ \ \ \ t=h(s).
\end{eqnarray}
Recall that $s=\phi(h(s))$. By applying the Fa\`a di Bruno formula again,
\begin{eqnarray}\label{formula:eq-01}
\phi^{(1)}(t)\cdot S_{1}^{1}h(s)=1,
\ \ \ \ \
\sum_{i=1}^{m}\phi^{(i)}(t)\cdot S_{m}^{i}h(s)=0 \ \mbox{ for }\ m=2,...,n+1.
\end{eqnarray}
Consider  \eqref{formula:eq-01} and solve for $\phi^{(i)}(t)$ in terms of $S_{m}^{l}h(s)$.
Then we get
\begin{eqnarray*}
\phi^{(i)}(t)=\frac{{\rm det}(\mathcal{M}_{n,i}h(s))}{{\rm det} (\mathcal{M}_{n}h(s))}, \ \ \ \ \ i=1,...,n+1.
\end{eqnarray*}
This together with  \eqref{formula:g-n} yields an expression of $g^{(n)}(s)$ as follows:
\begin{eqnarray}\label{formula-gn-1}
g^{(n)}(s)=\frac{1}{{\rm det}(\mathcal{M}_{n}h(s))}\sum_{i=1}^{n}{\rm det} (\mathcal{M}_{n,i+1}h(s)) \cdot S_{n}^{i}h(s)
\end{eqnarray}
for $s\in (-1,1)$ and $n=1,...,k$.

Note that  $g(\phi(t))=\dot \phi(t)$.
Then by the similar argument as above, we get
\begin{eqnarray*}
\sum_{i=1}^{m}g^{(i)}(s)\cdot S_{m}^{i}\phi(t)=\phi^{(m+1)}(t),\ \ \ \ \ m=1,...,n.
\end{eqnarray*}
Solving for $g^{(i)}(s)$ in terms of $S_{m}^{l}\phi(t)$,
we obtain
\begin{eqnarray*}
g^{(i)}(s)=\frac{{\rm det}(\mathcal{Q}_{n,i}\phi(t))}{{\rm det} (\mathcal{Q}_{n}\phi(t))},
\ \ \ \ \  i=1,...,n.
\end{eqnarray*}
This yields another expression of $g^{(n)}(s)$ as follows:
\begin{eqnarray}\label{formula-gn-2}
g^{(n)}(s)=\frac{{\rm det}(\mathcal{Q}_{n,n}\phi(t))}{{\rm det} (\mathcal{Q}_{n}\phi(t))},
\ \ \ \ \  n=1,...,k.
\end{eqnarray}
Therefore, the proof is finished by \eqref{formula-gn-1} and \eqref{formula-gn-2}.
\end{proof}

\begin{remark}
Let $\hat{\mathcal{M}}_{n}^{i,j}(s)$ ($1\leq i,j\leq n+1$)
denote the  cofactors of the matrices $\mathcal{M}_{n}h(s)$ for each $n=1,...,k$.
Then we obtain
\[
{\rm det}(\mathcal{M}_{n,i+1}h(s))=\hat{\mathcal{M}}_{n}^{1,i+1}(s).
\]
This together with \eqref{formula-gn-1} yields
\begin{eqnarray*}
g^{(n)}(s) \cdot {\rm det}(\mathcal{M}_{n}h(s))
   =\sum_{i=1}^{n}\hat{\mathcal{M}}_{n}^{1,i+1}(s)\cdot S_{n}^{i}h(s).
\end{eqnarray*}
Then we can compute
\begin{eqnarray*}
g^{(n)}(s)\cdot  {\rm det}(\mathcal{M}_{n}h(s))={\rm det}\left(
\begin{array}{cccc}
0 & S_{n}^{1}h(s) & \cdot\cdot\cdot & S_{n}^{n}h(s) \\
S_{2}^{1}h(s) & S_{2}^{2}h(s) & \cdot\cdot\cdot & 0\\
\cdot\cdot\cdot & \cdot\cdot\cdot & \cdot\cdot\cdot & \cdot\cdot\cdot\\
S_{n+1}^{1}h(s) & S_{n+1}^{2}h(s) & \cdot\cdot\cdot & S_{n+1}^{n+1}h(s)
\end{array}
\right).
\end{eqnarray*}
In the above matrix,
the vector $(0,S_{n}^{1}h(s),...,S_{n}^{n}h(s))$ replaces the first row of $\mathcal{M}_{n}h(s)$.
This formula provides another  way to compute $g^{(n)}(s)$.
\end{remark}

%\begin{remark}
%Applying the Fa\`a di Bruno formula to equation $t=h(\phi(t))$ yields
%\begin{eqnarray*}
%h^{(1)}(s)\cdot S_{1}^{1}\phi(t)=1,
%\ \ \ \ \
%\sum_{i=1}^{m}h^{(i)}(s)\cdot S_{m}^{i}\phi(t)=0 \ \mbox{ for }\ m=2,...,n+1.
%\end{eqnarray*}
%Solving for $h^{(i)}(s)$ in terms of $S_{m}^{l}\phi(t)$,
%we can obtain
%\begin{eqnarray}\label{df:h-i}
%h^{(i)}(s)=\frac{{\rm det}(\mathcal{M}_{n,i}\phi(t))}{{\rm det} (\mathcal{M}_{n}\phi(t))}, \ \ \ \ \ i=1,...,n+1.
%\end{eqnarray}
%Then by \eqref{formula:g-n} and \eqref{df:h-i},
%one is also able to present a certain condition,
%involving the limits of a function of $t$ as $t$ tends to $\pm \infty$,
%such that $G(s)$ is $C^{k}$-smooth on the interval $[-1,1]$.
%\end{remark}

Recall that $\mathcal{Q}_{n}\phi(t)$ ($n=1,...,k$) are defined by \eqref{matrix:Q-n}.
Let the $i$-th column of $\mathcal{Q}_{n}\phi(t)$ be replaced by the vector $(1,0,...,0)^{T}$.
Then we get
\begin{eqnarray}\label{matrix:Q-tilde-ni}
\tilde{\mathcal{Q}}_{n,i}\phi(t):=\left(
\begin{array}{cccccc}
S_{1}^{1}\phi(t) & 0 & \cdot\cdot\cdot &  1 & \cdot\cdot\cdot & 0 \\
S_{2}^{1}\phi(t) & S_{2}^{2}\phi(t) & \cdot\cdot\cdot & 0 & \cdot\cdot\cdot & 0 \\
\cdot\cdot\cdot & \cdot\cdot\cdot & \cdot\cdot\cdot & \cdot\cdot\cdot& \cdot\cdot\cdot & \cdot\cdot\cdot\\
S_{n}^{1}\phi(t) & S_{n}^{2}\phi(t) & \cdot\cdot\cdot& 0 & \cdot\cdot\cdot & S_{n}^{n}\phi(t)
\end{array}
\right),\ \ \ \ \ i=1,...,n.
\end{eqnarray}
Now we present some conditions such that
the function $F(x,s)$ is $C^{k}$-smooth on the extended domain $U\times [-1,1]$.

\begin{lemma}\label{lm:F-smooth}
Suppose that the transformation $\phi$ satisfies assumption {\bf (H1)}.
Let $h(s)=\phi^{-1}(s)$ for $s\in (-1,1)$.
Then $F(x,s)$ is $C^{k}$-smooth on the extended domain $U\times[-1,1]$ if the limits
\begin{eqnarray}\label{limt:F-trans-cond}
\lim_{t\to\pm \infty}
\sum_{l=1}^{n}\mu^{(l)}(t)\cdot \sum_{q(n,l)} n!\prod_{i=1}^{n}\frac{(\varphi_{n,i}(t))^{\lambda_{i}}}{(\lambda_{i}!)(i!)^{\lambda_{i}}}
\end{eqnarray}
exist for all $n=1,...,k$, where
\begin{eqnarray*}
\varphi_{n,i}(t):=\frac{{\rm det}(\tilde{\mathcal{Q}}_{n,i}\phi(t))}{{\rm det} (\mathcal{Q}_{n}\phi(t))},
\ \ \ \ \  i=1,...,n, \ \ t\in\mathbb{R}.
\end{eqnarray*}
\end{lemma}
\begin{proof}
By  assumption {\bf (H1)}, we can compute
\begin{eqnarray*}
\lim_{s\to \pm 1^{\mp}}F(x,s)=\lim_{s\to \pm 1^{\mp}}f(x,\mu(h(s)))
   =\lim_{t\to \pm \infty}f(x,\mu(t))=f(x,\mu^{\pm}).
\end{eqnarray*}
By the definition of $F(x,s)$,
we see that $F(x,s)$ is continuous at $s=\pm 1$.
Then $F(x,s)$ is $C^{k}$-smooth on the extended domain $U\times [-1,1]$ if
the left- and right-sided limits $s\to \pm 1^{\mp}$ exist
for all $n$-th partial derivatives, where $n=1,2,...,k$.

Now we use the Fa\`a di Bruno formula \eqref{FdB-formula}
to give the explicit expressions of all $n$-th ($n=1,2,...,k$) partial derivatives,
and then show that the conditions in this lemma
are sufficient for the $C^{k}$ smoothness of $F(x,s)$ on the extended domain $U\times [-1,1]$.
Set $\tilde{\mu}(s):=\mu(h(s))$ for $s\in (-1,1)$.
For each pair of integers $j$ and $n$ with $1\leq j+n\leq k$,
by the Fa\`a di Bruno formula \eqref{FdB-formula} we have
\begin{eqnarray}\label{eq:partial-d-F}
\frac{\partial^{j+n}}{\partial x^{j}\partial s^{n}}F(x,s)
=\sum_{l=1}^{n}\frac{\partial^{j+l}}{\partial x^{j}\partial \mu^{l}}f(x,\mu(h(s)))
\cdot \sum_{q(n,l)}n!\prod_{i=1}^{n}\frac{(\tilde{\mu}^{(i)}(s))^{\lambda_{i}}}{(\lambda_{i}!)(i!)^{\lambda_{i}}},
\ \ \ \ \ (x,s)\in U\times (-1,1),
\end{eqnarray}
where $q(n,l)$ are defined as below \eqref{df:S-n-l} and $\tilde{\mu}^{(i)}(s)=\frac{d^{i}}{ds^{i}}\tilde{\mu}(s)$.
Applying the Fa\`a di Bruno formula again yields that
for each $n$ with $1\leq n\leq k$,
\begin{eqnarray}\label{eq:d-mu}
\tilde{\mu}^{(n)}(s)
=\sum_{l=1}^{n}\frac{d^{l}}{dt^{l}}\mu(t)\cdot \sum_{q(n,l)}
  n!\prod_{i=1}^{n}\frac{(h^{(i)}(s))^{\lambda_{i}}}{(\lambda_{i}!)(i!)^{\lambda_{i}}}
=\sum_{l=1}^{n}\mu^{(l)}(t)\cdot S_{n}^{l}h(s),
\ \ \ \ \ s\in (-1,1),
\end{eqnarray}
where $S_{n}^{l}h(s)$ are defined by the formula \eqref{df:S-n-l}.
Noting that $t=h(\phi(t))$ for $t\in\mathbb{R}$,
we can compute
\begin{eqnarray*}
h^{(1)}(s)\cdot S_{1}^{1}\phi(t)=1,
\ \ \ \ \
\sum_{i=1}^{m}h^{(i)}(s)\cdot S_{m}^{i}\phi(t)=0 \ \mbox{ for }\ m=2,...,n.
\end{eqnarray*}
Solving for $h^{(i)}(s)$ in terms of $S_{m}^{l}\phi(t)$ yields
$h^{(i)}(s)=\varphi_{n,i}(t)$ for $i=1,...,n$,
where $\varphi_{n,i}(t)$ are defined in this lemma.
Substituting  $h^{(i)}(s)=\varphi_{n,i}(t)$ into \eqref{eq:d-mu} yields
\begin{eqnarray}\label{eq:mu-n}
\tilde{\mu}^{(n)}(s)
=\sum_{l=1}^{n}\mu^{(l)}(t)\cdot \sum_{q(n,l)} n!\prod_{i=1}^{n}\frac{(\varphi_{n,i}(t))^{\lambda_{i}}}{(\lambda_{i}!)(i!)^{\lambda_{i}}},
\ \ \ \ \ n=1,...,k.
\end{eqnarray}
Since $\mu(h(s))$ is continuous at $s=\pm 1$,
all partial derivatives $(\partial^{j+l}f(x,\mu(h(s)))/\partial x^{j}\partial \mu^{l})$
for $j$ and $l$ with $1\leq j+l\leq k$ are continuous at $s=\pm 1$.
Then by \eqref{eq:partial-d-F} and \eqref{eq:mu-n},
the left- and right-sided limits $s\to \pm 1^{\mp}$ exist
for each $(\partial^{j+n}F(x,s)/\partial x^{j}\partial s^{n})$ with $1\leq j+n\leq k$
if the limits \eqref{limt:F-trans-cond} exist.
This completes the proof.
\end{proof}

Next we give the transformation conditions for
the $C^{k}$ smoothness of the compactified system \eqref{eq:transf},
i.e., the functions $F(x,s)$ and $G(s)$ are $C^{k}$-smooth on the extended domain $U\times [-1,1]$.
By Lemmas \ref{lm:G-smooth} and \ref{lm:F-smooth}, one has the transformation conditions as follows:

\begin{theorem}\label{thm-transf-cond}
Consider the nonautonomous system \eqref{eq:NDS} and its compactified system \eqref{eq:comp}.
Suppose that the transformation $\phi$ satisfies assumption {\bf (H1)}. Let $h(s)=\phi^{-1}(s)$ for $s\in (-1,1)$.
Then the compactified system \eqref{eq:comp} is $C^{k}$-smooth on the extended domain $U\times [-1,1]$
if the limits \eqref{limt-G-cond} and \eqref{limt:F-trans-cond} exist.
\end{theorem}

To illustrate the conditions in the preceding theorem,
we give the explicit conditions for the $C^{1}$ and $C^{2}$ smoothness of
the compactified system \eqref{eq:comp} on the extended domain $U\times [-1,1]$.

\begin{corollary}\label{coro-1}
The following statements hold:
\begin{enumerate}
\item[{\bf (i)}]
The compactified system \eqref{eq:comp} is $C^{1}$-smooth on the extended domain $U\times[-1,1]$ if
\begin{eqnarray}
\lim_{s\to \pm 1^{\mp}}\left(-\frac{h''(s)}{(h'(s))^{2}}\right)
\!\!\!&=&\!\!\!
\lim_{t\to \pm \infty} \frac{\ddot{\phi}(t)}{\dot{\phi}(t)} \ \ \ \mbox{ exist},\label{cond-C-1-1}\\
\lim_{s\to \pm 1^{\mp}}\frac{d}{ds}\mu(h(s))
\!\!\!&=&\!\!\!
\lim_{t\to \pm \infty}\frac{\dot \mu(t)}{\dot \phi(t)} \ \ \ \mbox{ exist}.\label{cond-C-1-2}
\end{eqnarray}
\item[{\bf (ii)}]
The compactified system \eqref{eq:comp} is $C^{2}$-smooth on the extended domain $U\times[-1,1]$ if
\eqref{cond-C-1-1} and \eqref{cond-C-1-2} holds, and
\begin{eqnarray}
\lim_{s\to \pm 1^{\mp}}\left(\frac{2(h''(s))^{2}}{(h'(s))^{3}}-\frac{h^{(3)}(s)}{(h'(s))^{2}}\right)
\!\!\!&=&\!\!\!
\lim_{t\to \pm \infty} \left(\frac{\phi^{(3)}(t)}{(\dot{\phi}(t))^{2}}-\frac{(\ddot\phi(t))^{2}}{(\dot{\phi}(t))^{3}}\right) \ \ \ \mbox{ exist}, \label{cond-C-2-1}\\
\lim_{s\to \pm 1^{\mp}}\frac{d^{2}}{ds^{2}}\mu(h(s))
\!\!\!&=&\!\!\!
\lim_{t\to \pm \infty}\left(\frac{\ddot \mu(t)\dot \phi(t)-\dot\mu(t)\ddot \phi(t)}{(\dot \phi(t))^{3}}\right) \ \ \ \mbox{ exist}. \label{cond-C-2-2}
\end{eqnarray}
\end{enumerate}
\end{corollary}
\begin{proof}
By Lemmas \ref{lm:G-smooth} and \ref{lm:F-smooth},
we need to prove that the limits in \eqref{limt-G-cond} and \eqref{limt:F-trans-cond} with $n=1$
can be reduced to \eqref{cond-C-1-1} and \eqref{cond-C-1-2} respectively,
and with $n=2$ can be reduced to \eqref{cond-C-2-1} and \eqref{cond-C-2-2} respectively.
We prove only \eqref{cond-C-2-1} and \eqref{cond-C-2-2} and the proofs for the others are similar.

A direct computation yields
\begin{eqnarray*}
q(2,1)=\{(0,1)\},\ \ \  q(2,2)=\{(2,0)\},\ \ \
q(3,1)=\{(0,0,1)\},\ \ \ q(3,2)=\{(1,1,0)\},
\end{eqnarray*}
and
\begin{eqnarray*}
S_{2}^{1}h(s)=h''(s),\ \ \
S_{3}^{1}h(s)=h^{(3)}(s),\ \ \
S_{3}^{2}h(s)=3h'(s)h''(s),\ \ \
S_{i}^{i}h(s)=(h'(s))^{i}, \  \ \ \ \ i=1,2,3,
\end{eqnarray*}
from which  we obtain
\begin{eqnarray*}
{\rm det}(\mathcal{M}_{2,2}h(s))
\!\!\!&=&\!\!\!
{\rm det}\left(
\begin{array}{ccc}
S_{1}^{1}h(s) & 1 & 0 \\
S_{2}^{1}h(s) & 0 & 0\\
S_{3}^{1}h(s) & 0 & S_{3}^{3}h(s)
\end{array}
\right)=-S_{2}^{1}h(s)\cdot S_{3}^{3}h(s)=-(h'(s))^{3}h''(s),\\
{\rm det}(\mathcal{M}_{2,3}h(s))
\!\!\!&=&\!\!\!
{\rm det}\left(
\begin{array}{ccc}
S_{1}^{1}h(s) & 0 & 1 \\
S_{2}^{1}h(s) & S_{2}^{2}h(s) & 0\\
S_{3}^{1}h(s) & S_{3}^{2}h(s) & 0
\end{array}
\right)=S_{2}^{1}h(s)\cdot S_{3}^{2}h(s)-S_{2}^{2}h(s)\cdot S_{3}^{1}h(s)\\
\!\!\!&=&\!\!\!
3h'(s)(h''(s))^{2}-(h'(s))^2h^{(3)}(s).
\end{eqnarray*}
Then we get
\begin{eqnarray*}
\begin{aligned}
&  (\mathcal{M}_{2,2}h(s)) \cdot S_{2}^{1}h(s)+(\mathcal{M}_{2,3}h(s)) \cdot S_{2}^{2}h(s)\\
&\ \ \ =
-(h'(s))^{3}(h''(s))^{2}+\{3h'(s)(h''(s))^{2}-(h'(s))^2h^{(3)}(s)\}(h'(s))^{2}\\
&\ \ \ =
2(h'(s))^{3}(h''(s))^{2}-(h'(s))^{4}h^{(3)}(s).
\end{aligned}
\end{eqnarray*}
This together with
\begin{eqnarray*}
{\rm det}(\mathcal{M}_{2}h(s))=
{\rm det}\left(
\begin{array}{ccc}
S_{1}^{1}h(s) & 0 & 0\\
S_{2}^{1}h(s) & S_{2}^{2}h(s) & 0\\
S_{3}^{1}h(s) & S_{3}^{2}h(s) & S_{3}^{2}h(s)
\end{array}
\right)
=(h'(s))^{6}
\end{eqnarray*}
yields that
\begin{eqnarray}\label{formula:C-2-1}
\frac{1}{{\rm det}(\mathcal{M}_{2}h(s))}\sum_{i=1}^{2}{\rm det} (\mathcal{M}_{2,i+1}h(s)) \cdot S_{2}^{i}h(s)
=\frac{2(h''(s))^{2}}{(h'(s))^{3}}-\frac{h^{(3)}(s)}{(h'(s))^{2}}.
\end{eqnarray}
By \eqref{matrix:Q-n} and \eqref{matrix:Q-n-i}, we have
\begin{eqnarray}\label{formula:C-2-2}
\begin{aligned}
{\rm det}(\mathcal{Q}_{2}\phi(t))
&=
{\rm det}\left(
\begin{array}{cc}
S_{1}^{1}\phi(t) & 0   \\
S_{2}^{1}\phi(t) & S_{2}^{2}\phi(t)
\end{array}
\right)=(\dot{\phi}(t))^{3},\\
{\rm det}(\mathcal{Q}_{2,2}\phi(t))
&=
{\rm det}\left(
\begin{array}{cc}
S_{1}^{1}\phi(t) & \phi^{(2)}(t)  \\
S_{2}^{1}\phi(t) & \phi^{(3)}(t)
\end{array}
\right)=\dot{\phi}(t)\phi^{(3)}(t)-(\ddot\phi(t))^{2}.
\end{aligned}
\end{eqnarray}
Hence by \eqref{formula:C-2-1} and \eqref{formula:C-2-2},
the limits in \eqref{limt-G-cond} with $n=2$ is reduced to \eqref{cond-C-2-1}.

Now we consider \eqref{cond-C-2-2}.
By the definitions of $\mathcal{Q}_{n,}\phi(t)$ in \eqref{matrix:Q-n-i} and $\tilde{\mathcal{Q}}_{n,i}\phi(t)$ in \eqref{matrix:Q-tilde-ni},
we get
\begin{eqnarray*}
\varphi_{2,1}(t)
\!\!\!&=&\!\!\!
\frac{{\rm det}(\tilde{\mathcal{Q}}_{2,1}\phi(t))}{{\rm det} (\mathcal{Q}_{2}\phi(t))}
=\frac{(\dot \phi(t))^{2}}{(\dot \phi(t))^{3}}=\frac{1}{\dot \phi(t)},\\
\varphi_{2,2}(t)
\!\!\!&=&\!\!\!
\frac{{\rm det}(\tilde{\mathcal{Q}}_{2,2}\phi(t))}{{\rm det} (\mathcal{Q}_{2}\phi(t))}
=-\frac{\ddot \phi(t)}{(\dot \phi(t))^{3}}.
\end{eqnarray*}
Recalling that $q(2,1)=\{(0,1)\}$ and $q(2,2)=\{(2,0)\}$,
by the proof of Lemma \ref{lm:F-smooth} we get
\begin{eqnarray*}
\frac{d}{ds}\mu(h(s))
=\frac{\ddot \mu(t)\dot \phi(t)-\dot\mu(t)\ddot \phi(t)}{(\dot \phi(t))^{3}},\ \ \ \ t=h(s).
\end{eqnarray*}
Hence the limits in \eqref{limt:F-trans-cond} with  $n=2$ is reduced to \eqref{cond-C-2-2}.
This completes the proof.
\end{proof}

\subsubsection{A simple criterion for the $C^{k}$ extension}

In this subsection,
our goal is to further give a simple criterion on the time-dependent term $\mu$ for the $C^{k}$ extension.
We first introduce two auxiliary functions as follows:
\begin{eqnarray*}
\exp^{m}(t):=\underbrace{\exp(\exp(\cdots\exp(t)\cdots))}_{m\ \mbox{times}},
\ \ \ \ \
\ln^{m}(t):=\underbrace{\ln(\ln(\cdots\ln(t)\cdots))}_{m\ \mbox{times}},
\end{eqnarray*}
where $m$ is a positive integer.
For each integer $m\geq 1$, we define $\Phi_{m}(t)$ by
\begin{eqnarray*}
\Phi_{m}(t):=-(\ln^{m}(|t|))^{-1},
\end{eqnarray*}
whose domain $\mathcal{D}(\Phi_{m})$ is given by
\begin{eqnarray*}
\mathcal{D}(\Phi_{m})
=\left\{
\begin{array}{ll}
\{t\in\mathbb{R}: \ |t|>0\}, &\ \ \ \mbox{ if }\ m=1,\\
\{t\in\mathbb{R}: \ |t|>\exp^{m-1}(1)\}, &\ \ \ \mbox{ if }\ m\geq 2.
\end{array}
\right.
\end{eqnarray*}
By a direct computation, we get
\begin{eqnarray*}
\dot \Phi_{m}(t)
={\rm sgn}(t)(|t|\ln^{m}(|t|)\prod_{i=1}^{m}\ln^{i}|t|)^{-1}
\end{eqnarray*}
for $t\in \mathcal{D}(\Phi_{m})$,
where ${\rm sgn}(t)=1$ if $t>0$ and  ${\rm sgn}(t)=-1$ if $t<0$.
Then for each integer $m\geq 1$,
the function $\Phi_{m}$ satisfies the limits:
\begin{eqnarray}\label{limt:Phi-dotPhi}
\lim_{t\to \pm \infty}\Phi_{m}(t)=0,
\ \ \ \ \
\lim_{t\to \pm \infty}\dot \Phi_{m}(t)=0.
\end{eqnarray}
Furthermore, we have the next lemma.

\begin{lemma}\label{lm:Phi-m-prpty}
For each integer $m\geq 1$, the function $\Phi_{m}$ satisfies
\begin{eqnarray}\label{limts:frac}
\lim_{t\to\pm\infty}\frac{\Phi^{(n+1)}_{m}(|t|)}{(\dot\Phi_{m}(|t|))^{n}}=0
\end{eqnarray}
for every integer $n$ with $1\leq n\leq k$.
\end{lemma}

The proof of this lemma is given in Appendix B.

\begin{proposition}\label{prop-criterion}
Suppose that there exists a positive integer $m$ such that the limits
\begin{eqnarray}\label{limt:criterion-1}
\lim_{t\to\pm \infty}
\sum_{l=1}^{n}\mu^{(l)}(t)\cdot \sum_{q(n,l)} n!\prod_{i=1}^{n}\frac{(\tilde{\varphi}_{m,n,i}(t))^{\lambda_{i}}}{(\lambda_{i}!)(i!)^{\lambda_{i}}}
\end{eqnarray}
exist for all $n=1,...,k$, where
\begin{eqnarray*}
\tilde{\varphi}_{m,n,i}(t):=\frac{{\rm det}(\tilde{\mathcal{Q}}_{n,i}\Phi_{m}(t))}{{\rm det} (\mathcal{Q}_{n}\Phi_{m}(t))},
\ \ \ \ \  i=1,...,n, \ \ t\in\mathcal{D}(\Phi_{m}).
\end{eqnarray*}
Then there exists a transformation $\phi$ with the properties in {\bf (H1)} such that
the compactified system \eqref{eq:comp} is $C^{k}$-smooth on the extended domain $U\times[-1,1]$.
\end{proposition}
\begin{proof}
For a fixed integer $m\geq 1$, we define $\phi:\mathbb{R}\to (-1,1)$ by
\begin{eqnarray}\label{df:phi-crit}
\phi(t)=\left\{
\begin{array}{ll}
1+\Phi_{m}(t), &\ \ \ \mbox{ if }\ t>\exp^{m-1}(1)+1,\\
\tilde{\phi}_{m}(t), &\ \ \ \mbox{ if }\ -\exp^{m-1}(1)-1\leq t\leq \exp^{m-1}(1)+1,\\
-1-\Phi_{m}(t), &\ \ \ \mbox{ if }\ t<-\exp^{m-1}(1)-1,
\end{array}
\right.
\end{eqnarray}
where we set $\exp^{0}(1):=1$.
By \eqref{limt:Phi-dotPhi},
we can construct a function $\tilde{\phi}_{m}$ such that
$\phi$ is $C^{k+1}$ and strictly increasing on $\mathbb{R}$.
Then $\phi$ satisfies assumption {\bf (H1)}.

Recalling that ${\rm det}(\mathcal{Q}_{n}\phi(t))=\prod_{i=1}^{n}(\dot \phi(t))^{i}$,
we have
\begin{eqnarray*}
\frac{{\rm det}(\mathcal{Q}_{n,n}\phi(t))}{{\rm det} (\mathcal{Q}_{n}\phi(t))}
={\rm det}\left(
\begin{array}{cccccc}
S_{1}^{1}\phi(t)/\dot \phi(t) & 0 & \cdot\cdot\cdot & 0 & \phi^{(2)}(t)/\dot\phi(t) \\
S_{2}^{1}\phi(t)/(\dot \phi(t))^{2} & S_{2}^{2}\phi(t)/(\dot \phi(t))^{2} & \cdot\cdot\cdot & 0 & \phi^{(3)}(t)/(\dot \phi(t))^{2}  \\
\cdot\cdot\cdot & \cdot\cdot\cdot & \cdot\cdot\cdot & \cdot\cdot\cdot & \cdot\cdot\cdot \\
S_{n}^{1}\phi(t)/(\dot \phi(t))^{n} & S_{n}^{2}\phi(t)/(\dot \phi(t))^{n} & \cdot\cdot\cdot&  S_{n-1}^{n}\phi(t)/(\dot \phi(t))^{n} & \phi^{(n+1)}(t)/(\dot \phi(t))^{n}
\end{array}
\right).
\end{eqnarray*}
Then by Lemma \ref{lm:Phi-m-prpty},
\[
\lim_{t\to \pm \infty}\frac{{\rm det}(\mathcal{Q}_{n,n}\phi(t))}{{\rm det} (\mathcal{Q}_{n}\phi(t))}=0,
\ \ \ \ \ n=1,...,k.
\]
Using \eqref{limt:criterion-1} and  recalling the definition of $\phi$ in \eqref{df:phi-crit},
we obtain that
the limits \eqref{limt:F-trans-cond} exist for all $n=1,...,k$.
Therefore, the proof is finished by Theorem \ref{thm-transf-cond}.
\end{proof}

By a direction computation, we have a corollary of Proposition \ref{prop-criterion}.
\begin{corollary}
The following statements hold:
\begin{enumerate}
\item[{\bf (i)}]
The compactified system \eqref{eq:comp} is $C^{1}$-smooth on the extended domain $U\times[-1,1]$ if
there exists a positive integer $m$ such that the limits
\begin{eqnarray*}
\lim_{t\to \pm \infty}\frac{\dot \mu(t)}{\dot \Phi_{m}(t)}\ \ \  \mbox{ exist}.
\end{eqnarray*}

\item[{\bf (ii)}]
The compactified system \eqref{eq:comp} is $C^{2}$-smooth on the extended domain $U\times[-1,1]$ if
there exists a positive integer $m$ such that the limits
\begin{eqnarray*}
\lim_{t\to \pm \infty}\frac{\dot \mu(t)}{\dot \Phi_{m}(t)},
\ \ \ \ \
\lim_{t\to \pm \infty}\frac{\ddot \mu(t)\dot \Phi_{m}(t)-\dot\mu(t)\ddot \Phi_{m}(t)}{(\dot \Phi_{m}(t))^{3}}
\ \ \
\mbox{ exist}.
\end{eqnarray*}
\end{enumerate}
\end{corollary}

\subsection{One-sided compactifications}
\label{sec:one-side}

In this subsection,
we consider the cases that the time-dependent term $\mu$
is asymptotically constant with future limit $\mu^{+}$ or past limit $\mu^{-}$.
We aim to realize one-sided compactifications by the similar method as above.
No confusion should arise, we still use $\phi$ to denote the related transformations in the different cases.
We assume that the corresponding transformation $\phi$ satisfies the following:
\begin{enumerate}
\item[{\bf (H2)}] If $\mu$ is asymptotically constant with one-sided limit $\mu^{\pm}$,
i.e., future limit $\mu^{+}$ or past limit $\mu^{-}$,
then the transformation $\phi:\mathbb{R} \to \mathbb{R}$ is $C^{k+1}$-smooth on $\mathcal{I}_{\pm}$ and satisfies
\begin{eqnarray*}
\lim_{t\to \pm \infty}\phi(t)=\pm1,
\ \ \ \ \
\dot \phi(t)>0 \ \mbox{ for }\  t\in \mathcal{I}_{\pm},
\ \ \ \ \
\lim_{t\to \pm\infty}\dot \phi(t)=0,
\end{eqnarray*}
where $\mathcal{I}_{+}=[t_{+},+\infty)$ and $\mathcal{I}_{-}=(-\infty,t_{-}]$.
\end{enumerate}

Let $s=\phi(t)$, $s_{\pm}=\phi(t_{\pm})$ and
\begin{eqnarray*}
\mathcal{S}_{+}=[s_{+},1),
\ \ \ \ \
\mathcal{S}_{-}=(-1,s_{-}].
\end{eqnarray*}
Then we obtain one-sided compactified systems as follows:
\begin{eqnarray}\label{eq:comp+}
\begin{aligned}
\dot x &=F^{\pm}(x,s),\\
\dot s &=G^{\pm}(s),
\end{aligned}
\ \ \ \ \ \ \
(x,s)\in U\times \bar{\mathcal{S}}_{\pm},
\end{eqnarray}
where
\begin{eqnarray*}
F^{\pm}(x,s)\!\!\!&=&\!\!\!\left\{
\begin{aligned}
&f(x,\mu(h^{\pm}(s))), && \mbox{ if }\ s\in \mathcal{S}_{\pm},\\
&f(x,\mu^{\pm}), && \mbox{ if }\ s=\pm 1,
\end{aligned}
\right.
\\
G^{\pm}(s)\!\!\!&=&\!\!\!\left\{
\begin{aligned}
& \dot\phi(h^{\pm}(s)), && \mbox{ if }\ s\in \mathcal{S}_{\pm},\\
& 0, && \mbox{ if }\ s=\pm1,
\end{aligned}
\right.
\end{eqnarray*}
$h^{\pm}(s)=\phi^{-1}(s)$ for $s\in \mathcal{S}_{\pm}$
and $\bar{\mathcal{S}}_{\pm}$ denote the closures of $\mathcal{S}_{\pm}$.
The compactified system related to the case that $\mu$ is asymptotically constant with future limit $\mu^{+}$
(resp. past limit $\mu^{-}$) is called the right-sided (resp. left-sided) compactified system.
Similarly to the proof of Theorem \ref{thm-transf-cond},
we have the transformation conditions as follows.

\begin{theorem}\label{thm-transf-cond-pm}
Consider the nonautonomous system \eqref{eq:NDS} and its one-sided compactified system \eqref{eq:comp+}.
Suppose that the transformation $\phi$ satisfies assumption {\bf (H2)}.
Let $h^{\pm}(s)=\phi^{-1}(s)$ for $s\in\mathcal{S}_{\pm}$.
Then the following statements hold:

\begin{enumerate}
\item[{\bf (i)}]
The right-sided compactified system \eqref{eq:comp+} is $C^{k}$-smooth on the extended domain $U\times \bar{\mathcal{S}}_{+}$
if the limits
\begin{eqnarray*}
&& \lim_{s\to 1^{-}}\frac{1}{{\rm det}(\mathcal{M}_{n}h^{+}(s))}\sum_{i=1}^{n}{\rm det} (\mathcal{M}_{n,i+1}h^{+}(s)) \cdot S_{n}^{i}h^{+}(s)
=\lim_{t\to +\infty}\frac{{\rm det}(\mathcal{Q}_{n,n}\phi(t))}{{\rm det} (\mathcal{Q}_{n}\phi(t))},\\
&& \lim_{t\to +\infty}
\sum_{l=1}^{n}\mu^{(l)}(t)\cdot \sum_{q(n,l)} n!\prod_{i=1}^{n}\frac{(\varphi_{n,i}(t))^{\lambda_{i}}}{(\lambda_{i}!)(i!)^{\lambda_{i}}},
\end{eqnarray*}
exist for all $n=1,...,k$.

\item[{\bf (ii)}] The left-sided compactified system \eqref{eq:comp+} is $C^{k}$-smooth
on the extended domain $U\times \bar{\mathcal{S}}_{-}$ if the limits
\begin{eqnarray*}
&& \lim_{s\to -1^{+}}\frac{1}{{\rm det}(\mathcal{M}_{n}h^{-}(s))}\sum_{i=1}^{n}{\rm det} (\mathcal{M}_{n,i+1}h^{-}(s)) \cdot S_{n}^{i}h^{-}(s)
=\lim_{t\to -\infty}\frac{{\rm det}(\mathcal{Q}_{n,n}\phi(t))}{{\rm det} (\mathcal{Q}_{n}\phi(t))},\\
&& \lim_{t\to -\infty}
\sum_{l=1}^{n}\mu^{(l)}(t)\cdot \sum_{p(n,l)} n!\prod_{i=1}^{n}\frac{(\varphi_{n,i}(t))^{\lambda_{i}}}{(\lambda_{i}!)(i!)^{\lambda_{i}}},
\end{eqnarray*}
exist for all $n=1,...,k$.
\end{enumerate}
\end{theorem}

In the end of this section,
we remark that it is also able to  give the criteria for one-sided compactifications of the nonautonomous system \eqref{eq:NDS}
by the similar way in Proposition \ref{prop-criterion}.

\section{Invariant manifolds for compactified systems}
\label{sec:IM}

We first introduce some notations and definitions for this section.
The Jacobian matrix of a $C^{1}$ function  $h$ is denoted by  $Dh$,
$\|\cdot\|$ denotes the norm of an element in a Banach space
or the norm of a bounded linear operator,
$|\cdot|$ denotes the norm of a vector in an Euclidean space,
and $^{*}$ denotes the transpose of a vector.

We focus on the nonautonomous system \eqref{eq:NDS} with a bi-asymptotically constant term $\mu$.
When $\mu$ is asymptotically constant, i.e., it has a future or past limit,
the similar argument can be given by a simple modification.
Throughout this section, we also assume the condition as follows:
\begin{enumerate}
\item[{\bf (H3)}] There exists a positive integer $k\geq 2$ and a transformation $\phi$ with the properties in {\bf (H1)} such that
the limits \eqref{limt-G-cond} and \eqref{limt:F-trans-cond} exist,
and the following limit holds:
$$\lim_{t\to \pm \infty} \frac{\dot{\phi}(t)}{\phi(t)}=0.$$
\end{enumerate}
When assumption {\bf (H3)} holds,
by Theorem \ref{thm-transf-cond} we obtain that the compactified system \eqref{eq:comp} is $C^{k}$-smooth.
It is clear that the limit in {\bf (H3)} holds when $\phi$ decays slower than exponentially.
See also the example in section \ref{sec:example}.
If this limit holds, then we will see that each invariant set of the compactified system \eqref{eq:comp}
gains an additional center direction.
This leads to some trouble in analysis due to the nonuniqueness of center manifolds.
We are about to deal with it in this section.

Now we recall some topological notions for flows.
Let $\rho$ denote a metric in an Euclidean space.
We use $\Psi_{t}$ ($t\in\mathbb{R}$) to denote the flow generated by a vector field.
The $\omega$-limit set $\omega(\Lambda)$ and the $\alpha$-limit set $\alpha(\Lambda)$ of a set $\Lambda$ are defined by
\begin{eqnarray*}
\omega(\Lambda)):=\bigcap_{T>0}{\rm cl}\{\Psi_{t}(\Lambda): t>T\},
\ \ \ \ \
\alpha(\Lambda)):=\bigcap_{T<0}{\rm cl}\{\Psi_{t}(\Lambda): t<T\},
\end{eqnarray*}
where ${\rm cl}(\cdot)$ denote the closure operation and
$\Psi_{t}(\Lambda)$ denotes the flow evolution for time $t$ of the set $\Lambda$.
We call $\Lambda$ an invariant set of the flow $\Psi_{t}$ if $\Psi_{t}(\Lambda)=\Lambda$ for each $t\in\mathbb{R}$.

For a compact invariant set $\Lambda$,
we define its stable set $w^{\rm s}(\Lambda)$ and unstable set $w^{\rm u}(\Lambda)$, respectively, as
\begin{eqnarray*}
w^{\rm s}(\Lambda):=\{p: \omega(p)\subset \Lambda\},
\ \ \ \ \
w^{\rm u}(\Lambda):=\{p: \alpha(p)\subset \Lambda\},
\end{eqnarray*}
and it stable manifold $W^{\rm s}(\Lambda)$ and unstable manifold $W^{\rm u}(\Lambda)$, respectively, as
\begin{eqnarray*}
W^{\rm s}(\Lambda)
\!\!\!&:=&\!\!\!
\{p\in w^{\rm s}(\Lambda):\, \rho(\Psi_{t}(p),\Lambda)\leq Ke^{\lambda t} \ \ \ \mbox{for some}\ K>0,\,\lambda<0\ \mbox{ and } t>0 \},\\
W^{\rm u}(\Lambda)
\!\!\!&:=&\!\!\!
\{p\in w^{\rm u}(\Lambda):\, \rho(\Psi_{t}(p),\Lambda)\leq Ke^{\lambda t} \ \ \  \mbox{for some}\ K>0,\,\lambda>0\ \mbox{ and } t<0 \},
\end{eqnarray*}
where
\[
\rho(\Psi_{t}(p),\Lambda):=\inf_{\zeta\in\Lambda}\rho(\Psi_{t}(p),\zeta).
\]
The local versions of the stable/unstable sets and  manifolds
for the compact invariant set $\Lambda$ are defined as follows.
For an open set $V$ of the compact invariant set $\Lambda$,
the local stable/unstable sets $w^{\rm s,u}_{\rm loc}(\Lambda)$
and the local stable/unstable manifolds $W^{\rm s,u}_{\rm loc}(\Lambda)$ are given by
\begin{eqnarray*}
w^{\rm s,u}_{\rm loc}(\Lambda)
\!\!\!&:=&\!\!\!
\{p\in w^{s,u}(\Lambda):\, \Psi_{t}(p)\in V \ \mbox{ for all }\pm t>0\},\\
W^{\rm s,u}_{\rm loc}(\Lambda)
\!\!\!&:=&\!\!\!
\{p\in W^{s,u}(\Lambda):\, \Psi_{t}(p)\in V \ \mbox{ for all }\pm t>0\}.
\end{eqnarray*}

Invariant manifolds are useful in grasping the dynamics of the nonautonomous system \eqref{eq:NDS}.
Our goal is to investigate the invariant manifolds for compact invariant sets of the compactified system \eqref{eq:comp}.
According to the types of compact invariant sets,
the whole discussion is divided into three different cases.
The strategy of the following is as follows:

\begin{enumerate}
\item[$\bullet$]
Introduce the results on the invariant manifolds of equilibria,
which were proved by Wieczorek Xie and Jones \cite{Wieczorek-Xie-Jone-21} recently.

\item[$\bullet$]
Prove the results on the invariant manifolds for compact invariant sets,
which are normally hyperbolic invariant manifolds.

\item[$\bullet$]
Prove the results on the invariant manifolds for a class of compact invariant sets,
which are called admissible sets (see the  definition in subsection \ref{sec:Adm}).
\end{enumerate}
Actually, our method can be used to study more general compact invariant sets.

\subsection{Invariant manifolds of equilibria}
\label{sec:equiliria}

The simplest example of compact invariant sets for a flow is an equilibrium.
Let $\eta$ be an equilibrium of a vector field.
Suppose that the eigenvalues of the linearized system at $\eta$
consist of zero real-part, negative real-part and positive real-part eigenvalues,
whose related eigenspaces are denoted by  $E^{\rm c}, E^{\rm s}$ and $E^{\rm u}$, respectively.
A local center manifold $W_{\rm loc}^{\rm c}(\eta)$ of $\eta$ is given by the graph ${\rm gr}(h^{\rm c})$
of a function $h^{\rm c}: E^{\rm c}\cap V\to E^{\rm s}\oplus E^{\rm u}$,
which is flow-invariant relative to some neighborhood $V$ of $\eta$ and tangent to $E^{\rm c}$ at $\eta$.
A local center-stable manifold $W_{\rm loc}^{\rm cs}(\eta)$ of $\eta$ is given by the graph ${\rm gr}(h^{\rm cs})$
of a function $h^{\rm cs}: (E^{\rm c}\oplus E^{\rm s})\cap V\to E^{\rm u}$,
which is flow-invariant relative to some neighborhood $V$ of $\eta$ and tangent to $E^{\rm c}\oplus E^{\rm s}$ at $\eta$.

In particular,
suppose $\eta^{+}\in \{s=1\}$ and $\eta^{-}\in\{s=-1\}$ are equilibria of
the future limit system \eqref{eq:NDS-future} and the past limit system \eqref{eq:NDS-past}, respectively.
Then $\tilde{\eta}^{\pm}:=(\eta^{\pm},\pm1)$ become the equilibria of  the compactified system \eqref{eq:comp}.
By \eqref{cond-C-1-1}, {\bf (H1)} and {\bf (H3)}, the function $G$ in  \eqref{eq:comp} satisfies
\begin{eqnarray*}
G'(\pm 1)=\lim_{s\to \pm 1^{\mp}}(-\frac{h''(s)}{(h'(s))^{2}})
=\lim_{t\to \pm \infty} \frac{\ddot{\phi}(t)}{\dot{\phi}(t)}
=\lim_{t\to \pm \infty} \frac{\dot{\phi}(t)}{\phi(t)}=0,
\end{eqnarray*}
Note that  the Jacobian matrices of the compactified system \eqref{eq:comp} at $\tilde{\eta}^{\pm}$ are given by
\begin{eqnarray*}
\left(
\begin{array}{cc}
\frac{\partial f}{\partial x}(\eta^{\pm}, \mu^{\pm}) & \frac{\partial F}{\partial s}(\eta^{\pm},\pm 1)\\
0 & G'(\pm 1)
\end{array}
\right).
\end{eqnarray*}
Compared with the equilibria $\eta^{\pm}$ of the limit systems,
each $\tilde{\eta}^{\pm}$ of the compactified system \eqref{eq:comp} gains one center direction.
In general, the center  manifolds may be not unique \cite{Carr-81,Sijbran-85}.
Concerning the occurrence and uniqueness of the center and center-stable manifolds in the compactified system \eqref{eq:comp},
Wieczorek, Xie and Jones \cite{Wieczorek-Xie-Jone-21} recently proved the following statements.

\begin{theorem}\label{thm-equilbria-1}\cite[Theorem 3.2]{Wieczorek-Xie-Jone-21}
Suppose that assumption {\bf (H3)} holds and
$\eta^{+}$ is a hyperbolic saddle equilibrium of the future limit system \eqref{eq:NDS-future}.
Then $\tilde{\eta}^{+}=(\eta^{+},1)$ is a non-hyperbolic saddle of the compactified system \eqref{eq:comp},
its local center-stable manifold $W^{\rm cs}_{\rm loc}(\tilde{\eta}^{+})$
is unique, relatively to some neighborhood $\mathcal{V}$  of $\tilde{\eta}^{+}$,
and  $W^{\rm cs}_{\rm loc}(\tilde{\eta}^{+})=w_{\rm loc}^{\rm s}(\tilde{\eta}^{+})$
in the extended phase space of the compactified system \eqref{eq:comp}.
\end{theorem}

The proof of this theorem is based on a good local structure near an equilibrium.
Consider the compactified system \eqref{eq:comp} with time reversed.
Then one can also obtain the next theorem by applying Theorem \ref{thm-equilbria-1}.

\begin{theorem}\label{thm-equilbria-2}\cite[Theorem 3.4]{Wieczorek-Xie-Jone-21}
Suppose that assumption {\bf (H3)} holds and
$\eta^{-}$ is a hyperbolic sink of the past limit system \eqref{eq:NDS-past}.
Then $\tilde{\eta}^{-}=(\eta^{-},-1)$ is a non-hyperbolic saddle of the compactified system \eqref{eq:comp},
its local center manifold $W^{\rm c}_{\rm loc}(\tilde{\eta}^{-})$
is unique, relatively to some neighborhood $\mathcal{V}$  of $\tilde{\eta}^{-}$,
and  $W^{\rm c}_{\rm loc}(\tilde{\eta}^{-})=w_{\rm loc}^{\rm u}(\tilde{\eta}^{-})$
in the extended phase space of the compactified system \eqref{eq:comp}.
\end{theorem}

Besides equilibria, dynamical systems generated by differential equations
can possess more complicated compact invariant sets,
such as periodic orbits, tori and strange attractors
\cite{Arnold-Afrajmovich-Ilyashenko-13,Guck-Holmes-83,Robinson-99}.
Inspired by Theorems \ref{thm-equilbria-1} and \ref{thm-equilbria-2},
Wieczorek, Xie and Jones  \cite{Wieczorek-Xie-Jone-21} proposed a conjecture as follows:
\begin{conjecture}\label{conj:WXJ}
The statements for equilibria in Theorems \ref{thm-equilbria-1} and \ref{thm-equilbria-2} also hold
for general normally hyperbolic compact invariant sets of the future and past limit systems.
\end{conjecture}

In the next two subsections, we shall give a positive answer to this conjecture
for some more general invariant sets.

\subsection{Invariant manifolds for submanifolds}
\label{sec:NHIM}

In this subsection,
we state the results on the invariant manifolds for an compact invariant set,
which is a normally hyperbolic invariant manifold, including periodic orbits, tori, etc.
We also refer to \cite{Bates-Peter-Lu-98,Fenichel-71,Hirsh-Pugh-Shub-77,Wiggins-94}
for more information on normally hyperbolic invariant manifolds.

Consider a $C^{r}$ ($r\geq 1$) vector field on $\mathbb{R}^{N}$.
The flow generated by this vector field is still denoted by $\Psi_{t}$ ($t\in\mathbb{R}$).
Set
\[
\Phi(t;p):=D\Psi_{t}(p),\ \ \ \ \ p\in \Lambda.
\]
A compact invariant set $\Lambda$ of $\Psi_{t}$ is said to be normally hyperbolic
if the tangent bundle $T\mathbb{R}^{N}$ of $\mathbb{R}^{N}$ restricted to $\Lambda$ can be split into three continuous subbundles:
\[
T\mathbb{R}^{N}|_{\Lambda}=\mathcal{E}^{\rm s}\oplus T\Lambda\oplus\mathcal{E}^{\rm u},
\]
where $T\Lambda$ is the tangent bundle of $\Lambda$, such that
\begin{enumerate}
\item[(i)] The splitting is invariant under the linearized flow, i.e.,
\[
\Phi(t;p)\mathcal{E}^{i}_{p}=\mathcal{E}^{i}_{\Psi_{t}(p)},\ \ i=s,u,
\ \ \ \ \
\Phi(t;p)T_{p}\Lambda=T_{\Psi_{t}(p)}\Lambda
\ \ \ \mbox{ for each } p\in \Lambda \mbox{ and } t\in\mathbb{R}.
\]

\item[(ii)] There exist positive constants $K$, $\alpha$ and  $\beta$ with $\alpha<\beta$ such that
\begin{eqnarray*}
|\Phi(t;p)v|\!\!\!&\leq&\!\!\! K e^{-\beta t}|v|\ \ \   \mbox{ for all }\ p\in\Lambda, v\in\mathcal{E}^{\rm s}_{p}\ \mbox{ and }t\geq 0,\\
|\Phi(t;p)v|\!\!\!&\leq&\!\!\! K e^{\beta t}|v|\ \ \ \ \     \mbox{ for all }\ p\in\Lambda, v\in\mathcal{E}^{\rm u}_{p}\ \mbox{ and }t\leq 0,\\
|\Phi(t;p)v|\!\!\!&\leq&\!\!\! K e^{\alpha |t|}|v|\ \ \ \      \mbox{ for all }\ p\in\Lambda, v\in T\Lambda\
   \mbox{ and }t\in\mathbb{R}.
\end{eqnarray*}
\end{enumerate}
The subbundles $\mathcal{E}^{\rm s}$, $T\Lambda$ and $\mathcal{E}^{\rm u}$ are called
the stable bundle, the center bundle and the unstable bundle, respectively.
In particular,
if the tangent bundle $T\mathbb{R}^{N}$ has no unstable bundles,
then the invariant manifold $\Lambda$ is said to be normally stable.

Suppose that two compact invariant sets $\Lambda^{+}\subset \{s=1\}$ and $\Lambda^{-}\subset \{s=-1\}$
are the normally hyperbolic invariant manifolds of
the future limit system \eqref{eq:NDS-future} and the past limit system \eqref{eq:NDS-past}, respectively.
Then $\tilde{\Lambda}^{\pm}:=(\Lambda^{\pm},\pm1)$ become two compact invariant manifolds of the compactified system \eqref{eq:comp}.
Similarly to the case of equilibria,
we have that each $\tilde{\Lambda}^{\pm}$ gains an additional center direction
along the vector $(0,0,...,0,1)\in\mathbb{R}^{N+1}$,
which is independent of $x\in \Lambda^{\pm}$ and normal to $\{s=\pm 1\}$.
It follows that the tangent bundle $T\mathbb{R}^{N+1}$ restricted to each $\tilde{\Lambda}^{\pm}$ can be split into three continuous subbundles:
\begin{eqnarray}\label{eq:direct}
T\mathbb{R}^{N+1}|_{\tilde{\Lambda}^{\pm}}=\mathcal{E}^{\rm s}_{\pm}\oplus \mathcal{E}^{\rm c}_{\pm} \oplus\mathcal{E}^{\rm u}_{\pm},
\end{eqnarray}
where each $\mathcal{E}^{\rm c}_{\pm}$ is the direct sum of  $T\Lambda^{\pm}$ and the additional center bundle.
Let $P^{\pm}_{\rm cs}(p)$ and $P^{\pm}_{\rm u}(p)$ be the projections of $\mathbb{R}^{N+1}$
on $\mathcal{E}^{\rm s}_{\pm}(p)\oplus \mathcal{E}^{\rm c}_{\pm}(p)$ along $\mathcal{E}^{\rm u}_{\pm}(p)$
and on $\mathcal{E}^{\rm u}_{\pm}(p)$ along $\mathcal{E}^{\rm s}_{\pm}(p)\oplus \mathcal{E}^{\rm c}_{\pm}(p)$,
respectively.

The local center-stable manifolds of $\tilde{\Lambda}^{+}$
and the local center manifolds of $\tilde{\Lambda}^{-}$ are defined as follows.
A local center-stable manifold $W^{\rm cs}_{\rm loc}(\tilde{\Lambda}^{+})$ of $\tilde{\Lambda}^{+}$
is a flow-invariant manifold, relative to some neighborhood $U$ of $\tilde{\Lambda}^{+}$,
and  is tangent to the bundle $\mathcal{E}^{\rm s}_{+}\oplus \mathcal{E}^{\rm c}_{+}$ along $\tilde{\Lambda}^{+}$.
A local center manifold $W^{\rm c}_{\rm loc}(\tilde{\Lambda}^{-})$ of $\tilde{\Lambda}^{-}$
is a flow-invariant manifold, relative to some neighborhood $U$ of $\tilde{\Lambda}^{-}$,
and  is tangent to the bundle $\mathcal{E}^{\rm c}_{-}$ along $\tilde{\Lambda}^{-}$.
We shall extend the results of Theorems \ref{thm-equilbria-1} and \ref{thm-equilbria-2}
to normally hyperbolic invariant manifolds.

Before that, we first make some preparations in the following.
Let $\bar{\Omega}=\Omega \cup \partial \Omega$ be a $C^{r}$ $(r\geq 1)$
compact submanifolds with boundary contained in $\mathbb{R}^{N+1}$,
where $\partial \Omega$ denotes the boundary of $\Omega$.
The submanifold $\bar{\Omega}$ is said to be overflowing invariant under a vector field $\mathcal{F}_{0}$,
whose flow is still denoted by $\Psi_{t}$ ($t\in\mathbb{R}$),
if for each $p\in \bar{\Omega}$, its backward orbit  lies in $\bar{\Omega}$,
and the vector field $\mathcal{F}_{0}$ points strictly outward on $\partial\Omega$.
One can write $T\mathbb{R}^{N+1}|_{\Omega}=T\Omega\oplus \mathcal{N}$,
where $\mathcal{N}$ is the normal bundle of $\Omega$.
Let $\Pi$ denote the projection on $\mathcal{N}$ corresponding to this direct sum.
For every $p\in \Omega$, $v_{0}\in T_{p}\Omega$ and $w_{0}\in \mathcal{N}_{p}$, let
\begin{eqnarray*}
v(t)=\Phi(t;p)v_{0},
\ \ \ \ \
w(t)=\Pi(\Phi(t;p)w_{0}),
\ \ \ \ \ t\leq 0.
\end{eqnarray*}
We define the generalized Lyapunov-type numbers of the overflowing invariant manifold $\Omega$ as follows:
\begin{eqnarray*}
\alpha(p):=\inf\left\{a>0:\,a^{t}/|w(t)|\to 0\ \mbox{ as }t\to -\infty \mbox{ for all } w_{0}\in \mathcal{N}_{p}\right\}.
\end{eqnarray*}
If $\alpha(p)<1$, then we define
\begin{eqnarray*}
\sigma(p):=\inf\left\{\sigma:\,|v(t)|/|w(t)|^{\sigma}\to 0\ \mbox{ as }t\to -\infty
\mbox{ for all } v_{0}\in T_{p}\Omega, w_{0}\in \mathcal{N}_{p}\right\}.
\end{eqnarray*}

By Fenichel's theorem on the persistence of overflowing invariant manifolds (see \cite[Theorem 1]{Fenichel-71}),
we can prove the following perturbation results for the overflowing invariant manifold $\bar{\Omega}$.

\begin{lemma}\label{thm-Fenichel-1}
Let $\bar{\Omega}=\Omega\cup\partial \Omega$ be a $C^{k}$ compact manifolds with boundary
and overflowing invariant under the vector field $\mathcal{F}_{0}$.
Suppose that
\begin{eqnarray*}
\alpha(p)<1,\ \ \ \sigma(p)<1,\ \ \ \mbox{ for all }\, p\in \Omega.
\end{eqnarray*}
Then for each $C^{k}$ vector field $\mathcal{F}_{\rm pert}$, $C^{1}$-close to $\mathcal{F}_{0}$,
there exists a  $C^{k}$ manifold  $\bar{\Omega}_{\rm pert}$ overflowing invariant under the vector field $\mathcal{F}_{\rm pert}$
and $C^{k}$ diffeomorphic to $\bar{\Omega}$.
\end{lemma}

We shall use this lemma to establish the existence of the local center-stable manifold $W_{\rm loc}^{\rm cs}(\tilde{\Lambda}^{+})$.

\begin{lemma}\label{prop-NHIM-1}
Suppose that assumption {\bf (H3)} holds
and the future limit system \eqref{eq:NDS-future} has a compact invariant set $\Lambda^{+}$,
which is a normally hyperbolic invariant manifold.
Then $\tilde{\Lambda}^{+}=(\Lambda^{+},1)$ has a local $C^{k}$-smooth center-stable manifold $W_{\rm loc}^{\rm cs}(\tilde{\Lambda}^{+})$,
relative to some neighborhood $V$ of $\tilde{\Lambda}^{+}$,
and $W^{\rm cs}_{\rm loc}(\tilde{\Lambda}^{+})\subset w_{\rm loc}^{\rm s}(\tilde{\Lambda}^{+})$
in the extended phase space of the compactified system \eqref{eq:comp}.
\end{lemma}
\begin{proof}
To prove this lemma,
we only need to prove that this lemma holds for the invariant manifold $(\Lambda^{+},0)$ of
the following auxiliary system:
\begin{eqnarray}\label{eq:mod-1}
\dot x =\tilde{F}(x,s),
\ \ \
\dot s =\tilde{G}(s),
\end{eqnarray}
such that
\begin{eqnarray*}
\tilde{F}(x,s)=\left\{
\begin{aligned}
&F(x,s+1), && \mbox{ for }\ -2\leq s\leq 0,\\
&F(x,1-s), && \mbox{ for }\ 0<s<s_{0},
\end{aligned}
\right.
\ \ \
\tilde{G}(s)=\left\{
\begin{aligned}
& G(s+1), && \mbox{ for }\ -2\leq s\leq 0,\\
& -G(1-s), && \mbox{ for }\ 0<s<s_{0},
\end{aligned}
\right.
\end{eqnarray*}
where $0<s_{0}<1$ and $\tilde{G}(0)=\tilde{G}'(0)=0$.
Let $U_{1}\subset U$ be a bounded open neighborhood of the invariant manifold $\Lambda^{+}$ in $\{s=0\}$.
Consider system \eqref{eq:mod-1} with time reversed, i.e.,
\begin{eqnarray}\label{eq:mod-2}
\dot x =-\tilde{F}(x,s),
\ \ \
\dot s =-\tilde{G}(s).
\end{eqnarray}
Then we can choose a sufficiently small $\epsilon$ with $0<\epsilon<s_{0}/2$ such that
system \eqref{eq:mod-2} can be arbitrarily $C^{1}$-close to
\begin{eqnarray}\label{eq:mod-3}
\dot x =-f(x,\mu^{+}),
\ \ \
\dot s =-\tilde{G}(s),
\end{eqnarray}
uniformly in the set $V_{1}:=\{(x,s):\,x\in \bar{U}_{1},\,|s|\leq \epsilon\}$.

We now construct a manifold with the overflowing property for system \eqref{eq:mod-3}.
Recall that $\Lambda^{+}$ is a normally hyperbolic invariant manifold of the future limit system \eqref{eq:NDS-future}.
Then by Fenichel's theorem \cite[Theorem 4]{Fenichel-71},
there exists a $C^{k}$ local stable manifold $W^{\rm s}_{\rm loc}(\Lambda^{+})$ of $\Lambda^{+}$,
relative to the neighborhood $U_{1}$ of $\Lambda^{+}$, for the future limit system \eqref{eq:NDS-future},
where $W^{\rm s}_{\rm loc}(\Lambda^{+})$ contains $\Lambda^{+}$ and is tangent to $\mathcal{E}^{\rm s}_{+}\oplus T\Lambda^{+}$.
By the theory of isolating block \cite{Conley-Easton-71},
there exists a neighborhood $U_{2}$ of $\Lambda^{+}$ inside $W^{\rm s}_{\rm loc}(\Lambda^{+})$
such that $\bar{U}_{2}=U_{2}\cup \partial U_{2}$ is $C^{k}$ and overflowing invariant
under the flow of the future limit system \eqref{eq:NDS-future} with time reversed, i.e., the first equation in \eqref{eq:mod-3}.
Note that two equations in \eqref{eq:mod-3} are independent,
and $-\tilde{G}(s)<0$ for $-\epsilon\leq s<0$ and $-\tilde{G}(s)>0$ for $0<s\leq \epsilon$.
Then the $C^{k}$ compact submanifold $\bar{U}_{2}\times \{|s|\leq \epsilon\}$ of $\mathbb{R}^{N+1}$
is overflowing invariant under system \eqref{eq:mod-3}.

To finish the proof, we now compute the generalized Lyapunov-type numbers.
By the normal hyperbolicity of $\Lambda^{+}$,
the normal direction of the overflowing invariant manifold
$\bar{U}_{2}\times \{|s|\leq \epsilon\}$ of  system \eqref{eq:mod-3} is exponentially stable.
Then $\alpha(p)<1$ for all $p\in \bar{U}_{2}\times \{|s|\leq \epsilon\}$.
Since $\tilde{G}'(0)=0$, we can further compute $\sigma(p)<1$ for all $p\in \bar{U}_{2}\times \{|s|\leq \epsilon\}$.
By Lemma \ref{thm-Fenichel-1},
there exists a sufficiently small $\epsilon_{0}>0$ and a $C^{k}$ manifold
$W^{\rm cs}_{\rm loc}((\Lambda^{+},0))$
such that $W^{\rm cs}_{\rm loc}((\Lambda^{+},0))$ contains $(\Lambda^{+},0)$,
and is $C^{k}$ diffeomorphic to $\bar{U}_{2}\times \{|s|\leq \epsilon\}$
for $0<\epsilon\leq \epsilon_{0}$, overflowing invariant under system \eqref{eq:mod-2}
and tangent to $\mathcal{E}^{\rm s}_{+}\oplus \mathcal{E}^{\rm c}_{+}$.
Consequently, $W^{\rm cs}_{\rm loc}((\Lambda^{+},0))$
is a local center-stable manifold of $(\Lambda^{+},0)$ for system \eqref{eq:mod-1}.
It follows that $W^{\rm cs}_{\rm loc}((\Lambda^{+},0))\subset w^{\rm s}((\Lambda^{+},0))$.
This finishes the proof.
\end{proof}

For convenience, we write the vector field governed by the compactified system \eqref{eq:comp} as $\mathcal{F}(X)$,
where $X=(x,s)^{*}$. Then we can write the compactified system \eqref{eq:comp} into a simple form
\begin{eqnarray}\label{eq:simple-form}
\dot X=\mathcal{F}(X).
\end{eqnarray}
By assumption {\bf (H3)} and Theorem \ref{thm-transf-cond},
we see that system \eqref{eq:simple-form} is at least $C^{2}$-smooth.
With no confusion, we still use $\Psi_{t}$ ($t\in\mathbb{R}$) to denote the corresponding flow of this system.

\begin{lemma}\label{lm:est-1}
Let $X_{1}$ and $X_{2}$ be the solutions of system  \eqref{eq:simple-form}
with the initial conditions $X(0)=X_{1}(0)$ and $X(0)=X_{2}(0)$, respectively,
where the set $V$ is defined  as in Lemma \ref{prop-NHIM-1}.
Suppose that $X_{1}(t)$ and $X_{2}(t)$ stay in the set $V$ for all $t\geq 0$.
Then
\begin{eqnarray*}
|X_{2}(t)-X_{1}(t)|\leq |X_{2}(0)-X_{1}(0)| \exp(M_{1}t),\ \ \ \ \ t\geq 0,
\end{eqnarray*}
where $M_{1}$ is independent of the initial conditions $X_{1}(0)$ and $X_{2}(0)$,
and satisfies
$$M_{1}=\max_{X\in \bar{V}}\|D\mathcal{F}(X)\|.$$
\end{lemma}
\begin{proof}
Since $X_{i}(t)\in V$ for all $t\geq 0$ and $i=1,2$,
by the Mean Value Theorem we have
\[
|\mathcal{F}(X_{2}(t))-\mathcal{F}(X_{1}(t))|
\leq \max_{X\in \bar{V}}\|D\mathcal{F}(X)\|\cdot |X_{2}(t)-X_{1}(t)|
=M_{1}|X_{2}(t)-X_{1}(t)|
\]
for each $t\geq 0$. This together with
\[
X_{2}(t)-X_{1}(t)=X_{2}(0)-X_{1}(0)+\int_{0}^{t}(\mathcal{F}(X_{2}(\tau))-\mathcal{F}(X_{1}(\tau)))d\tau
\]
yields
\begin{eqnarray*}
|X_{2}(t)-X_{1}(t)|
\!\!\!&\leq&\!\!\!
 |X_{2}(0)-X_{1}(0)|+\int_{0}^{t}|\mathcal{F}(X_{2}(\tau))-\mathcal{F}(X_{1}(\tau))|d\tau\\
\!\!\!&\leq&\!\!\!
 |X_{2}(0)-X_{1}(0)|+M_{1}\int_{0}^{t}|X_{2}(\tau)-X_{1}(\tau)|d\tau, \ \ \ \ t\geq 0.
\end{eqnarray*}
Therefore, the proof is finished by the Gronwall inequality.
\end{proof}

By the Tubular Neighborhood Theorem,
we can choose $V$ in Lemma \ref{prop-NHIM-1} as a neighborhood of the local center-manifold $W^{\rm cs}_{\rm loc}(\tilde{\Lambda}^{+})$
such that for each $\xi\in V$, there exists a unique point $\xi_{\rm cs}\in W^{\rm cs}_{\rm loc}(\tilde{\Lambda}^{+})$ satisfying
\begin{eqnarray*}
\rho(\xi,\xi_{\rm cs})=\rho(\xi,W^{\rm cs}_{\rm loc}(\tilde{\Lambda}^{+})).
\end{eqnarray*}
Define $\Pi_{\rm cs}:V\to W^{\rm cs}_{\rm loc}(\tilde{\Lambda}^{+})$ and $\Pi_{\rm u}:V\to V$ by
\begin{eqnarray*}
\Pi_{\rm cs}(\xi)=\xi_{\rm cs},
\ \ \ \ \
 \Pi_{\rm u}(\xi)=\xi-\xi_{\rm cs},
\ \ \ \ \ \xi\in V,
\end{eqnarray*}
and $\mathcal{P}:W^{\rm cs}_{\rm loc}(\tilde{\Lambda}^{+}) \to \Lambda^{+}$ by
\begin{eqnarray*}
\mathcal{P}(\xi)=\zeta,
\ \ \
\rho(\xi,\zeta)=\rho(\xi,\Lambda^{+}),
\ \ \ \ \ \xi\in W^{\rm cs}_{\rm loc}(\tilde{\Lambda}^{+}).
\end{eqnarray*}
Since the invariant set $\Lambda^{+}$ is compact, the map $\mathcal{P}$ is well-defined.
By the Whitney Embedding Theorem \cite{Whitney-36},
we can assume that the maps $\Pi_{\rm cs}$ and $\Pi_{\rm u}$ are $C^{k}$ ($k\geq 2$).
In particular, we have
\[
D\Pi_{\rm cs}(\zeta)=P_{\rm cs}(\zeta),
\ \ \ \ \
D\Pi_{\rm u}(\zeta)=P_{\rm u}(\zeta),
\ \ \ \ \ \zeta \in \Lambda^{+},
\]
where $P_{\rm cs}(\cdot):=P^{+}_{\rm cs}(\cdot)$ and $P_{\rm u}(\cdot):=P^{+}_{\rm u}(\cdot)$ are defined below \eqref{eq:direct}.

\begin{lemma}\label{lm:prj}
For each point $\xi\in V$,
let $\xi_{\rm cs}=\Pi_{\rm cs}(\xi)$ and $\zeta=\mathcal{P}(\xi_{\rm cs})$.
Then there exists a sufficiently small $\varepsilon_{0}>0$ such that
for each $\varepsilon$ with $0<\varepsilon<\varepsilon_{0}$ and each $\xi\in V$
with $\rho(\xi,\zeta)<\varepsilon$ and $\rho(\xi_{\rm cs},\zeta)<\varepsilon$,
\begin{eqnarray*}
\frac{|\omega(0)|}{2}\leq |P_{\rm u}(\zeta)\omega(0)|\leq \frac{3|\omega(0)|}{2},
\ \ \
|P_{\rm cs}(\zeta)\omega(0)|\leq \frac{|\omega(0)|}{2},
\ \ \
\omega(0)=\xi-\xi_{\rm cs}.
\end{eqnarray*}
\end{lemma}
\begin{proof}
Note that $\Pi_{\rm cs}(\xi_{\rm cs})=\xi_{\rm cs}$ and $\Pi_{\rm u}(\xi_{\rm cs})=0$.
Then
\begin{eqnarray}\label{prj-u}
\begin{aligned}
\omega(0)&=\Pi_{\rm u}(\xi)\\
   &=\Pi_{\rm u}(\xi)-\Pi_{\rm u}(\xi_{\rm cs})\\
   &=P_{\rm u}(\zeta)\omega(0)+(\Pi_{\rm u}(\xi)-\Pi_{\rm u}(\xi_{\rm cs})-D\Pi_{\rm u}(\zeta)\omega(0)).
\end{aligned}
\end{eqnarray}
By the Mean Value Theorem,
there exists a constant $\lambda$ with $0\leq \lambda\leq 1$ such that
\begin{eqnarray*}
\Pi_{\rm u}(\xi)-\Pi_{\rm u}(\xi_{\rm cs})-D\Pi_{\rm u}(\zeta)\omega(0)
=(D\Pi_{\rm u}(\lambda\xi+(1-\lambda)\xi_{\rm cs})-D\Pi_{\rm u}(\zeta))\omega(0).
\end{eqnarray*}
Applying the Mean Value Theorem again,
for each $\xi\in V$ with $\rho(\xi,\zeta)<\varepsilon$ and $\rho(\xi_{\rm cs},\zeta)<\varepsilon$, we have
\begin{eqnarray*}
|\Pi_{\rm u}(\xi)-\Pi_{\rm u}(\xi_{\rm cs})-D\Pi_{\rm u}(\zeta)\omega(0)|
\leq \varepsilon M_{2}|\omega(0)|,
\end{eqnarray*}
where $M_{2}:=\sup_{\xi\in \bar{V}}\|D^{2}\Pi_{u}(\xi)\|$.
It follows that  there exists a small $\varepsilon_{0}>0$ such that
for each $\varepsilon$ with $0<\varepsilon<\varepsilon_{0}$,
\begin{eqnarray*}
|\Pi_{\rm u}(\xi)-\Pi_{\rm u}(\xi_{\rm cs})-D\Pi_{\rm u}(\zeta)\omega(0)|
\leq \frac{1}{2}|\omega(0)|.
\end{eqnarray*}
This together with \eqref{prj-u} and $\omega(0)=P_{\rm u}(\zeta)\omega(0)+P_{\rm cs}(\zeta)\omega(0)$ yields
\begin{eqnarray*}
&&|P_{\rm u}(\zeta)\omega(0)|
\leq |\omega(0)|+|\Pi_{\rm u}(\xi)-\Pi_{\rm u}(\xi_{\rm cs})-D\Pi_{\rm u}(\zeta)\omega(0)|
\leq \frac{3}{2}|\omega(0)|,\\
&&|P_{\rm u}(\zeta)\omega(0)|
\geq |\omega(0)|-|\Pi_{\rm u}(\xi)-\Pi_{\rm u}(\xi_{\rm cs})-D\Pi_{\rm u}(\zeta)\omega(0)|
\geq \frac{1}{2}|\omega(0)|,
\end{eqnarray*}
and
\begin{eqnarray*}
|P_{\rm cs}(\zeta)\omega(0)|
\!\!\!&=&\!\!\!
|\omega(0)-P_{\rm u}(\zeta)\omega(0)|\\
\!\!\!&=&\!\!\!
|\Pi_{\rm u}(\xi)-\Pi_{\rm u}(\xi_{\rm cs})-D\Pi_{\rm u}(\zeta)\omega(0)|\\
\!\!\!&\leq &\!\!\! \frac{1}{2}|\omega(0)|.
\end{eqnarray*}
This finishes the proof.
\end{proof}

For each point $\xi\in V$,
let $\xi(t)$, $\xi_{\rm cs}(t)$ and $\zeta(t)$ to denote the solutions of system  \eqref{eq:simple-form}
with the initial conditions $X(0)=\xi$, $X(0)=\Pi_{\rm cs}(\xi)=\xi_{\rm cs}$ and $X(0)=\mathcal{P}(\xi_{\rm cs})=\zeta$, respectively.
Set
$$
\omega(t):=\xi(t)-\xi_{\rm cs}(t),\ \ \ \ \ t\geq 0.
$$
Then we have the following two lemmas.

\begin{lemma}\label{lm:est-2}
Suppose that $\rho(\xi,\zeta)<\varepsilon$ and $\rho(\xi_{\rm cs},\zeta)<\varepsilon$
for a small $\varepsilon$ with $0<\varepsilon<\varepsilon_{0}$,
and $\xi(t)$ stay in the set $V$ for all $t$ with $0\leq t\leq T$ ,
where $\varepsilon_{0}$ is defined in Lemma \ref{lm:prj}.
Then there exists a constant $M(T)>0$, dependent of $T$, such that
\begin{eqnarray*}
|P_{\rm cs}(\zeta(t))\omega(t)|
\leq Ke^{\alpha t}|P_{\rm cs}(\zeta)\omega(0)|+\varepsilon M(T)|\omega(0)|e^{\alpha t},
\ \ \ \ \ 0\leq t\leq T.
\end{eqnarray*}
\end{lemma}
\begin{proof}
Note that
\[
P_{\rm cs}(\zeta(t))\Phi(t;\zeta)=\Phi(t;\zeta)P_{\rm cs}(\zeta).
\]
This together with
\[
\Phi(t;\zeta)'=D\mathcal{F}(\zeta(t))\Phi(t;\zeta)
\]
yields
\begin{eqnarray*}
\{P_{\rm cs}(\zeta(t))\Phi(t;\zeta)\}'
=
P_{\rm cs}(\zeta(t))'\Phi(t;\zeta)+P_{\rm cs}(\zeta(t))D\mathcal{F}(\zeta(t))\Phi(t;\zeta),
\end{eqnarray*}
and
\begin{eqnarray*}
\{\Phi(t;\zeta)P_{\rm cs}(\zeta)\}'
=
D\mathcal{F}(\zeta(t))\Phi(t;\zeta)P_{\rm cs}(\zeta)
=D\mathcal{F}(\zeta(t))P_{\rm cs}(\zeta(t))\Phi(t;\zeta).
\end{eqnarray*}
It follows that
\begin{eqnarray*}
P_{\rm cs}(\zeta(t))'
=D\mathcal{F}(\zeta(t))P_{\rm cs}(\zeta(t))-P_{\rm cs}(\zeta(t))D\mathcal{F}(\zeta(t)).
\end{eqnarray*}
Using the above equation, we get
\begin{eqnarray*}
\{P_{\rm cs}(\zeta(t))\omega(t)\}'
=D\mathcal{F}(\zeta(t))P_{\rm cs}(\zeta(t))\omega(t)+P_{\rm cs}(\zeta(t))\Delta_{1}(\xi)(t),
\end{eqnarray*}
where
\begin{eqnarray*}
\Delta_{1}(\xi)(t):=\mathcal{F}(\xi(t))-\mathcal{F}(\xi_{\rm cs}(t))-D\mathcal{F}(\zeta(t))\omega(t).
\end{eqnarray*}
Then by the Variation of Constants Formula, we get
\begin{eqnarray}\label{formula-Pcs-w}
P_{\rm cs}(\zeta(t))\omega(t)
=\Phi(t;\zeta)P_{\rm cs}(\zeta)\omega(0)+\int_{0}^{t}\Phi(t;\zeta)\Phi^{-1}(\tau;\zeta)\Delta_{1}(\xi)(\tau)d\tau.
\end{eqnarray}
By the Mean Value Theorem,
there exists a function $\lambda(t)$ with $0\leq \lambda(t)\leq 1$ for $0\leq t\leq T$  such that
\[
\mathcal{F}(\xi(t))-\mathcal{F}(\xi_{\rm cs}(t))=D\mathcal{F}(\lambda(t)\xi(t)+(1-\lambda(t))\xi_{\rm cs}(t))\omega(t).
\]
Since $\rho(\xi,\zeta)<\varepsilon$, $\rho(\xi_{\rm cs},\zeta)<\varepsilon$, applying Lemma \ref{lm:est-1} yields that
there is a constant $\tilde{M}(T)>0$, dependent of $T$, such that
\begin{eqnarray}\label{est:Delta-1}
|\Delta_{1}(\xi)(t)|\leq \varepsilon \tilde{M}(T)|\omega(0)|,
\ \ \ \ \ 0\leq t\leq T.
\end{eqnarray}
Then by \eqref{formula-Pcs-w}, \eqref{est:Delta-1} and the definition of normal hyperbolicity, we have
\begin{eqnarray*}
|P_{\rm cs}(\zeta(t))\omega(t)|
\!\!\!&\leq&\!\!\!
 Ke^{\alpha t}|P_{\rm cs}(\zeta)\omega(0)|+\int_{0}^{t}Ke^{\alpha (t-\tau)}\cdot\varepsilon \tilde{M}(T)|\omega(0)| d\tau\\
\!\!\!&\leq&\!\!\!
 Ke^{\alpha t}|P_{\rm cs}(\zeta)\omega(0)|+\frac{\varepsilon K \tilde{M}(T)}{\alpha}|\omega(0)|e^{\alpha t}.
\end{eqnarray*}
This finishes the proof.
\end{proof}

\begin{lemma}\label{lm:est-3}
Suppose that $\rho(\xi,\zeta)<\varepsilon$ and $\rho(\xi_{\rm cs},\zeta)<\varepsilon$
for a small $\varepsilon$ with $0<\varepsilon<\varepsilon_{0}$,
and $\xi(t)$  stay in the set $V$ for all $t$ with $0\leq t\leq T$,
where $\varepsilon_{0}$ is defined in Lemma \ref{lm:prj}.
Then
\begin{eqnarray*}
|P_{\rm u}(\zeta(T))\omega(T)|
\geq |P_{\rm u}(\zeta)\omega(0)|e^{\beta T}K^{-1}(1+\delta(\varepsilon,T)),
\end{eqnarray*}
where $\delta(\varepsilon,T)$ satisfies that for each $T>0$,
$\delta(\varepsilon,T)\to 0$ as $\varepsilon\to0$.
\end{lemma}
\begin{proof}
Similarly to the proof of Lemma \ref{lm:est-2},
we can check
\begin{eqnarray*}
P_{\rm u}(\zeta(T-t))'
=-D\mathcal{F}(\zeta(T-t))P_{\rm u}(\zeta(T-t))+P_{\rm u}(\zeta(T-t))D\mathcal{F}(\zeta(T-t)).
\end{eqnarray*}
Then we have
\begin{eqnarray*}
\{P_{\rm u}(\zeta(T-t))\omega(T-t)\}'
\!\!\!&=&\!\!\!-D\mathcal{F}(\zeta(T-t))P_{\rm u}(\zeta(T-t))\omega(T-t)+P_{\rm u}(\zeta(T-t))\Delta_{2}(\xi)(T-t),
\end{eqnarray*}
where
\begin{eqnarray*}
\Delta_{2}(\xi)(T-t):=-\mathcal{F}(\xi(T-t))+\mathcal{F}(\xi_{\rm cs}(T-t))+D\mathcal{F}(\zeta(T-t))\omega(T-t).
\end{eqnarray*}
By  the Variation of Constants Formula,
\begin{eqnarray*}
P_{\rm u}(\zeta(T-t))\omega(T-t)
\!\!\!&=&\!\!\!
\Phi(T-t;\zeta)\Phi^{-1}(T;\zeta)P_{\rm u}(\zeta(T))\omega(T)\\
\!\!\!& &\!\!\!
+\int_{0}^{t}\Phi(T-t;\zeta)\Phi^{-1}(T-\tau;\zeta)\Delta_{2}(\xi)(T-\tau)d\tau.
\end{eqnarray*}
Since $\rho(\xi,\zeta)<\varepsilon$ and $\rho(\xi_{\rm cs},\zeta)<\varepsilon$,
by the Mean Value Theorem and Lemma \ref{lm:est-1},
there exists a constant $M(T)>0$, dependent of $T$, such that
\begin{eqnarray*}
|\Delta_{2}(\xi)(T-t)|\leq \varepsilon M(T)|\omega(T-t)|,\ \ \ \ \ 0\leq t\leq T.
\end{eqnarray*}
This together with the definition of normal hyperbolicity and Lemma \ref{lm:est-2} yields
\begin{eqnarray*}
|P_{\rm u}(\zeta(T-t))\omega(T-t)|
\!\!\!&\leq&\!\!\!
       Ke^{-\beta t}|P_{\rm u}(\zeta(T))\omega(T)|
       +\varepsilon KM(T)\int_{0}^{t}e^{-\beta (t-\tau)}|P_{\rm u}(\zeta(T-\tau))\omega(T-\tau)|d\tau\\
\!\!\!& &\!\!\!
       +\varepsilon KM(T)\int_{0}^{t}e^{-\beta (t-\tau)}|P_{\rm cs}(\zeta(T-\tau))\omega(T-\tau)|d\tau.
\end{eqnarray*}
By the above inequality, Lemmas \ref{lm:prj} and \ref{lm:est-2},
we can choose $M(T)>0$ such that
\begin{eqnarray*}
|P_{\rm u}(\zeta(T-t))\omega(T-t)|e^{\beta t}
\!\!\!&\leq&\!\!\!
       K|P_{\rm u}(\zeta(T))\omega(T)|+\varepsilon M(T)e^{(\beta-\alpha)t}|P_{\rm u}(\zeta)\omega(0)|\\
\!\!\!& &\!\!\!
       +\varepsilon KM(T)\int_{0}^{t}e^{\beta \tau}|P_{\rm u}(\zeta(T-\tau))\omega(T-\tau)|d\tau.
\end{eqnarray*}
Then by the Gronwall inequality, we get
\begin{eqnarray*}
\ \
|P_{\rm u}(\zeta(T-t))\omega(T-t)|e^{\beta t}
\leq  \left\{K|P_{\rm u}(\zeta(T))\omega(T)|+\varepsilon M(T)e^{(\beta-\alpha)t}|P_{\rm u}(\zeta)\omega(0)|\right\}
       \exp(\varepsilon KM(T) t).
\end{eqnarray*}
After setting $t=T$, we can compute
\begin{eqnarray*}
|P_{\rm u}(\zeta(T))\omega(T)|
\geq |P_{\rm u}(\zeta)\omega(0)|e^{\beta T}K^{-1}\left\{e^{-\varepsilon KTM(T)}-\varepsilon M(T)e^{-\alpha T}\right\}.
\end{eqnarray*}
Then the proof is finished by a direct computation.
\end{proof}

Now we state the main results on the compact invariant manifolds $\tilde{\Lambda}^{\pm}$.

\begin{theorem}\label{thm-NHIM-1}
Suppose that assumption {\bf (H3)} holds
and the future limit system \eqref{eq:NDS-future} has a compact invariant set $\Lambda^{+}$,
which is a normally hyperbolic invariant manifold.
Then $\tilde{\Lambda}^{+}=(\Lambda^{+},1)$ is a normally non-hyperbolic invariant manifold of the compactified system \eqref{eq:comp},
its local center-stable manifold $W_{\rm loc}^{\rm cs}(\tilde{\Lambda}^{+})$ is unique,
relative to some neighborhood $\mathcal{V}$ of $\tilde{\Lambda}^{+}$,
and $W^{\rm cs}_{\rm loc}(\tilde{\Lambda}^{+})=w_{\rm loc}^{\rm s}(\tilde{\Lambda}^{+})$
in the extended phase space of the compactified system \eqref{eq:comp}.
\end{theorem}
\begin{proof}
It is clear that $\tilde{\Lambda}^{+}$ is a normally non-hyperbolic invariant manifold of the compactified system \eqref{eq:comp}.
By Lemma \ref{prop-NHIM-1},
we get that $W_{\rm loc}^{\rm cs}(\tilde{\Lambda}^{+})$ is the local center-stable manifold of $\tilde{\Lambda}^{+}$.
Now we prove the uniqueness of $W_{\rm loc}^{\rm cs}(\tilde{\Lambda}^{+})$,
relative to a neighborhood $\mathcal{V}$ of $\tilde{\Lambda}^{+}$,
and $w_{\rm loc}^{\rm s}(\tilde{\Lambda}^{+})\subset W^{\rm cs}_{\rm loc}(\tilde{\Lambda}^{+})$.

By the definition of normal hyperbolicity, we see that $K\geq 1$.
Fix a constant $T$ with $T>\beta^{-1}\ln (2K)$. By Lemma \ref{lm:est-3},
we can choose a small constant $\varepsilon$ with $0<\varepsilon<\varepsilon_{0}$ such that
\begin{eqnarray*}
(2K)^{-1}e^{\beta T}(1+\delta(\varepsilon,T))-\varepsilon M_{2} \exp(2M_{1}T)=:\theta>1,
\end{eqnarray*}
where $M_{2}=\sup_{\xi\in \bar{V}}\|D^{2}\Pi_{u}(\xi)\|$.
Define $\mathcal{V}$ by
\begin{eqnarray*}
\mathcal{V}:=\left\{\xi\in V: \rho(\xi,\zeta)<\varepsilon,\ \rho(\xi_{\rm cs},\zeta)<\varepsilon\right\},
\end{eqnarray*}
where $\xi_{\rm cs}=\Pi_{\rm cs}(\xi)$ and $\zeta=\mathcal{P}(\xi_{\rm cs})$.
Note that $\tilde{\Lambda}^{+}\subset W_{\rm loc}^{\rm cs}(\tilde{\Lambda}^{+})$.
Then we see that
\begin{eqnarray*}
\xi=\Pi_{\rm cs}(\xi)=\xi_{\rm cs},
\ \ \
\xi=\mathcal{P}(\xi)=\zeta,\ \ \ \ \ \xi\in \tilde{\Lambda}^{+}.
\end{eqnarray*}
This yields $\rho(\xi,\zeta)=0$ and $\rho(\xi_{\rm cs},\zeta)=0$.
Hence $\mathcal{V}$ is a neighborhood of $\tilde{\Lambda}^{+}$.

Suppose that $\xi\in \mathcal{V}$ satisfies that $\xi(t)=\Psi_{t}(\xi)\in \mathcal{V}$ for all $t\geq 0$.
Set
\[
\xi_{k}:=\xi(kT),
\ \ \
\xi_{cs,k}:=\Pi_{\rm cs}(\xi_{k}),
\ \ \
\zeta_{k}:=\mathcal{P}(\xi_{cs,k})
\]
for all integers $k\geq 0$. Then
\begin{eqnarray}\label{ineq:ite-k}
\rho(\xi_{k},\zeta_{k})<\varepsilon,
\ \ \
\rho(\xi_{cs,k},\zeta_{k})<\varepsilon,
\ \ \ k\geq 0.
\end{eqnarray}
By the Mean Value Theorem, we have that
\begin{eqnarray*}
\Pi_{\rm u}(\xi_{k})
\!\!\!&=&\!\!\!
\Pi_{\rm u}(\Psi_{T}(\xi_{k-1}))-\Pi_{\rm u}(\Psi_{T}(\xi_{cs,k-1}))\\
\!\!\!&=&\!\!\!
D\Pi_{\rm u}(\tilde{\lambda}(k-1)\Psi_{T}(\xi_{k-1})+(1-\tilde{\lambda}(k-1))\Psi_{T}(\xi_{cs,k-1}))(\Psi_{T}(\xi_{k-1})-\Psi_{T}(\xi_{cs,k-1}))
\end{eqnarray*}
for some $\tilde{\lambda}(k-1)$ with $0\leq \tilde{\lambda}(k-1)\leq 1$ and $k\geq 1$.
This together with \eqref{ineq:ite-k}, the Mean Value Theorem and Lemma \ref{lm:est-1} yields that
\begin{eqnarray*}
|\Pi_{\rm u}(\xi_{k})-P_{u}(\Psi_{T}(\zeta_{k-1}))(\Psi_{T}(\xi_{k-1})-\Psi_{T}(\xi_{cs,k-1}))|
\leq \varepsilon M_{2} \exp(2M_{1}T)|\xi_{k-1}-\xi_{cs,k-1}|
\end{eqnarray*}
for each $k\geq 1$. Then by Lemmas \ref{lm:prj} and \ref{lm:est-3}, we get
\begin{eqnarray*}
|\Pi_{\rm u}(\xi_{k})|
\!\!\!&\geq&\!\!\!
|P_{u}(\Psi_{T}(\zeta_{k-1}))(\Psi_{T}(\xi_{k-1})-\Psi_{T}(\xi_{cs,k-1}))|-\varepsilon M_{2} \exp(2M_{1}T)|\xi_{k-1}-\xi_{cs,k-1}|\\
\!\!\!&\geq&\!\!\!
|P_{u}(\zeta_{k-1})(\xi_{k-1}-\xi_{cs,k-1})|e^{\beta T}K^{-1}(1+\delta(\varepsilon,T))-\varepsilon M_{2} \exp(2M_{1}T)|\xi_{k-1}-\xi_{cs,k-1}|\\
\!\!\!&=&\!\!\!
\{(2K)^{-1}e^{\beta T}(1+\delta(\varepsilon,T))-\varepsilon M_{2} \exp(2M_{1}T)\}|\xi_{k-1}-\xi_{cs,k-1}|\\
\!\!\!&=&\!\!\!  \theta |\xi_{k-1}-\xi_{cs,k-1}|.
\end{eqnarray*}
Noting that $\Pi_{\rm u}(\xi_{k-1})=\xi_{k-1}-\xi_{cs,k-1}$, we have
\begin{eqnarray*}
|\Pi_{\rm u}(\xi_{k})|\geq \theta |\Pi_{\rm u}(\xi_{k-1})|,\ \ \ \ \ k\geq 1.
\end{eqnarray*}
It follows that
\begin{eqnarray}\label{ineq:iterat}
|\Pi_{\rm u}(\xi_{k})|\geq \theta^{k} |\Pi_{\rm u}(\xi)|,
\ \ \ \ \ k\geq 1.
\end{eqnarray}
If $|\Pi_{\rm u}(\xi)|\neq 0$, then by \eqref{ineq:iterat},
then there exists an integer $k_{0}\geq 1$ such that $|\Pi_{\rm u}(\xi(kT))|>\varepsilon$.
This contradicts the assumption that $\xi(t)=\Psi_{t}(\xi)\in \mathcal{V}$ for all $t\geq 0$.
Thus, $|\Pi_{\rm u}(\xi)|=0$ if $\xi(t)=\Psi_{t}(\xi)\in \mathcal{V}$ for all $t\geq 0$.
This implies the uniqueness of $W_{\rm loc}^{\rm cs}(\tilde{\Lambda}^{+})$
and $w_{\rm loc}^{\rm s}(\tilde{\Lambda}^{+})\subset W^{\rm cs}_{\rm loc}(\tilde{\Lambda}^{+})$.
Therefore, the proof is now complete.
\end{proof}

\begin{remark}
The similar argument in the proof of Theorem \ref{thm-NHIM-1}
can be also used to study the uniqueness of the perturbed overflowing invariant manifolds as in Lemma \ref{thm-Fenichel-1}
and the stable or unstable manifolds of slow manifolds in singularly perturbed systems \cite{Fenichel-79}.
\end{remark}

\begin{theorem}\label{thm-NHIM-2}
Suppose that assumption {\bf (H3)} holds
and the past limit system \eqref{eq:NDS-past} has a compact invariant set $\Lambda^{-}$,
which is a normally stable invariant manifold.
Then $\tilde{\Lambda}^{-}=(\Lambda^{-},1)$ is a normally non-hyperbolic invariant manifold of the compactified system \eqref{eq:comp},
its local center manifold $W_{\rm loc}^{\rm c}(\tilde{\Lambda}^{-})$ is unique,
relative to some neighborhood $\mathcal{V}$  of $\tilde{\Lambda}^{-}$,
and $W^{\rm c}_{\rm loc}(\tilde{\Lambda}^{+})=w_{\rm loc}^{\rm u}(\tilde{\Lambda}^{-})$
in the extended phase space of the compactified system \eqref{eq:comp}.
\end{theorem}
\begin{proof}
It is clear that $\tilde{\Lambda}^{-}$ is a normally non-hyperbolic invariant manifold of the compactified system \eqref{eq:comp}.
Consider the compactified system \eqref{eq:comp} with time reversed.
Then by Theorem \ref{thm-NHIM-1}, we can obtain $W^{cu}_{\rm loc}(\tilde{\Lambda}^{-})=w_{\rm loc}^{\rm u}(\tilde{\Lambda}^{-})$.
Note that $\Lambda^{-}$ is a normally stable invariant manifold of system \eqref{eq:NDS-past}.
Then $\tilde{\Lambda}^{-}$ has no unstable normal directions.
Hence we further have $W^{\rm c}_{\rm loc}(\tilde{\Lambda}^{-})=W^{cu}_{\rm loc}(\tilde{\Lambda}^{-})=w_{\rm loc}^{\rm u}(\tilde{\Lambda}^{-})$.
This finishes the proof.
\end{proof}

\subsection{Invariant manifolds for admissible sets}
\label{sec:Adm}

In this subsection,
we consider the invariant manifolds for a class of general compact invariant sets,
which are called admissible sets \cite{Chow-Liu-Yi-00}.

Now we introduce the precise definition of admissible sets.
The connected compact invariant sets $\Lambda^{\pm}$ of the limit systems \eqref{eq:NDS-future} and \eqref{eq:NDS-past}
are said to be admissible sets if there exist constants $K_{i}^{\pm}>0$ ($i=1,2,3$) such that
for each $\varepsilon>0$, there are smooth neighborhoods $U^{\pm}$ of the invariant sets $\Lambda^{\pm}$
on which the following three conditions hold for all $x\in \partial U^{\pm}$:
\begin{enumerate}
\item[{\rm (i')}] $K^{\pm}_{1} \varepsilon\leq \rho(x,\Lambda^{\pm})\leq \varepsilon$.

\item[{\rm (ii')}] $|(f(x,\mu^{\pm}),v(x))|\leq K^{\pm}_{2}\varepsilon$,
where $(\cdot,\cdot)$ denotes the inner product of two vectors
and $v(x)$ denotes a unit normal vector to $\partial U^{\pm}$ at $x$.

\item[{\rm (iii')}] In a neighborhood $U_{x}$ of $x$,
if $\partial U^{\pm}\cap U_{x}$ is the graph of a function
$h:T_{x}\partial U^{\pm}\to T^{\bot}_{x}\partial U^{\pm}$,
then $|\xi^{*}{\rm Hess}(h)(x)\xi|\leq K_{3}^{\pm}\varepsilon^{-1}$ for each unit vector $\xi\in T_{x}\partial U^{\pm}$.
\end{enumerate}
The admissible sets include a large class of compact invariant sets, such as invariant manifolds, homoclinic or heteroclinic loops, etc.

Let $\Lambda^{\pm}$ be the admissible sets of the limit systems \eqref{eq:NDS-future} and \eqref{eq:NDS-past}.
To construct the local center-stable manifolds of $\tilde{\Lambda}^{\pm}:=(\Lambda^{\pm},\pm 1)$,
we also require an additional topological property for the admissible sets $\Lambda^{\pm}$.
For a fixed $x\in \Lambda^{\pm}$,
a subspace  $E(x)$ of $T_{x}\mathbb{R}^{N}$ is called a generalized tangent space of $\Lambda^{\pm}$ at $x$
if for $x_{1}$ and $x_{2}$ in $\Lambda^{\pm}$ with $x_{1}\neq x_{2}$,
\begin{eqnarray*}
\rho\left(\frac{x_{1}-x_{2}}{|x_{1}-x_{2}|},E(x)\right) \to 0\ \ \ \mbox{ as }\ \ (x_{1},x_{2})\to (x,x).
\end{eqnarray*}
Two continuous subbundles $E(\Lambda^{\pm}):=\bigcup\{E(x)\subset T_{x}\mathbb{R}^{N} \mbox { for } x\in \Lambda^{\pm}\}$ of the tangent bundle $T_{\Lambda^{\pm}}\mathbb{R}^{N}$ of $\mathbb{R}^{N}$
over $\Lambda^{\pm}$ are called generalized tangent bundles of $\Lambda^{\pm}$
if $E(x)$ is a generalized tangent space of $\Lambda^{\pm}$ at each $x\in \Lambda^{\pm}$.

Concerning the admissible sets, we have the following two theorems.

\begin{theorem}\label{thm-admset-1}
Let $\Lambda^{+}$ be a normally hyperbolic admissible set of the future limit system \eqref{eq:NDS-future}.
Suppose that assumption {\bf (H3)} holds and the center subbundle $T\Lambda^{+}$ is a generalized tangent bundle of $\Lambda^{+}$.
Then $\tilde{\Lambda}^{+}=(\Lambda^{+},1)$ is a normally non-hyperbolic admissible set of the compactified system \eqref{eq:comp},
its local center-stable manifold $W_{\rm loc}^{\rm cs}(\tilde{\Lambda}^{+})$ is unique,
relative to some neighborhood $\mathcal{V}$ of $\tilde{\Lambda}^{+}$,
and $W^{\rm cs}_{\rm loc}(\tilde{\Lambda}^{+})=w_{\rm loc}^{\rm s}(\tilde{\Lambda}^{+})$
in the extended phase space of the compactified system \eqref{eq:comp}.
\end{theorem}
\begin{proof}
It is clear that $\tilde{\Lambda}^{+}$ is a normally non-hyperbolic admissible set of the compactified system \eqref{eq:comp}.
Since the center subbundle $T\Lambda^{+}$ is a generalized tangent bundle of $\Lambda^{+}$,
the center subbundle $\mathcal{E}^{\rm c}_{+}$ is also a generalized tangent bundle of $\tilde{\Lambda}^{+}$.
Then by  \cite[Proposition 1]{Chow-Liu-Yi-00},
there exists a $C^{k}$ manifold $\hat{W}_{\rm loc}^{\rm cs}(\tilde{\Lambda}^{+})$
such that $\tilde{\Lambda}^{+} \subset \hat{W}_{\rm loc}^{\rm cs}(\tilde{\Lambda}^{+})$ and
$T_{\xi}\hat{W}_{\rm loc}^{\rm cs}(\tilde{\Lambda}^{+})=\mathcal{E}^{\rm s}_{+}(\xi)\oplus \mathcal{E}^{\rm c}_{+}(\xi)$
for each $\xi\in \tilde{\Lambda}^{+}$.
Then we can construct a bundle structure in a neighborhood of $\tilde{\Lambda}^{+}$.
By applying \cite[Proposition 2]{Chow-Liu-Yi-00},
the admissible set $\tilde{\Lambda}^{+}$ has a local center-stable manifold $W_{\rm loc}^{\rm cs}(\tilde{\Lambda}^{+})$,
which is $C^{k}$ and $T_{\xi}W_{\rm loc}^{\rm cs}(\tilde{\Lambda}^{+})=\mathcal{E}^{\rm s}_{+}(\xi)\oplus \mathcal{E}^{\rm c}_{+}(\xi)$
for each $\xi\in \Lambda^{+}$.
Similarly to Theorem \ref{thm-NHIM-1},
we can prove the uniqueness of the local center-stable manifold  $W_{\rm loc}^{\rm cs}(\tilde{\Lambda}^{+})$
and $W^{\rm cs}_{\rm loc}(\tilde{\Lambda}^{+})=w_{\rm loc}^{\rm s}(\tilde{\Lambda}^{+})$.
This finishes the proof.
\end{proof}

Consider the compactified system \eqref{eq:comp} with time reversed.
Then by Theorem \ref{thm-admset-1}, we can prove the following statements.

\begin{theorem}\label{thm-admset-2}
Let $\Lambda^{-}$ be a normally stable admissible set of the past limit system \eqref{eq:NDS-past}.
Suppose that assumption {\bf (H3)} holds and the center subbundle $T\Lambda^{-}$ is a generalized tangent bundle of $\Lambda^{-}$.
Then $\tilde{\Lambda}^{-}=(\Lambda^{-},1)$ is a normally non-hyperbolic admissible set of the compactified system \eqref{eq:comp},
its local center manifold $W_{\rm loc}^{\rm c}(\tilde{\Lambda}^{-})$ is unique,
relative to some neighborhood $\mathcal{V}$  of $\tilde{\Lambda}^{-}$,
and $W^{\rm c}_{\rm loc}(\tilde{\Lambda}^{+})=w_{\rm loc}^{\rm u}(\tilde{\Lambda}^{-})$
in the extended phase space of the compactified system \eqref{eq:comp}.
\end{theorem}

\begin{remark}
We have proved the uniqueness of the local center-stable manifold $W_{\rm loc}^{\rm cs}(\tilde{\Lambda}^{+})$
for a normally hyperbolic invariant manifold or admissible set $\Lambda^{+}$ of the future limit system \eqref{eq:NDS-future}
when embedded in the extended phase space,
and the uniqueness of the local center manifold $W_{\rm loc}^{\rm c}(\tilde{\Lambda}^{-})$
for a normally stable invariant manifold or admissible set $\Lambda^{-}$ of the past limit system \eqref{eq:NDS-past}
when embedded in the extended phase space.
We also believe that for more general invariant sets $\Lambda^{\pm}$,
the uniqueness of the corresponding manifold $W_{\rm loc}^{\rm cs}(\tilde{\Lambda}^{+})$ or $W_{\rm loc}^{\rm c}(\tilde{\Lambda}^{-})$
can be treat by small modifications of the present argument.
\end{remark}

\section{Example and remark}
\label{sec:example}

In order to demonstrate the  theory developed in this paper,
we consider a mechanical model
\begin{eqnarray}\label{app-ode}
\ddot u+r(t,u)\dot u+u=0
\end{eqnarray}
with a time-dependent resistance
\[
r(t,u)=-\left(\frac{2}{\pi}\arctan t\right)^{2}(1-u^{2}).
\]
This mechanical model can be recast as a system
\begin{eqnarray}\label{app-sys-1}
\begin{aligned}
\dot u &= v,\\
\dot v &=(\mu(t))^{2}(1-u^{2})v-u,
\end{aligned}
\end{eqnarray}
where
\begin{eqnarray*}
\mu(t)=\frac{2}{\pi}\arctan t, \ \ \ \ \ t\in\mathbb{R}.
\end{eqnarray*}
It is clear that the nonautonomous term $\mu$ satisfies the limits:
\begin{eqnarray*}
\lim_{t\to +\infty}\mu(t)=1, \ \ \ \ \
\lim_{t\to -\infty}\mu(t)=-1.
\end{eqnarray*}
Then $\mu$ is bi-asymptotically constant with future limit $1$ and past limit $-1$.
Letting $t\to \pm \infty$ yields that the autonomous future and past limit systems of system \eqref{app-sys-1} are given by
\begin{eqnarray}\label{app-sys-limit}
\begin{aligned}
\dot u &= v,\\
\dot v &=(1-u^{2})v-u.
\end{aligned}
\end{eqnarray}

To realize the compactification of system \eqref{app-sys-1}, we set $\phi(t):=\mu(t)=\frac{2}{\pi}\arctan t$.
We can check that
\begin{eqnarray*}
\dot \phi(t)=\frac{2}{\pi}\cdot \frac{1}{t^{2}+1}>0 \mbox{ for } t\in\mathbb{R},\ \ \ \
\lim_{t\to \pm \infty}\dot \phi(t)=\lim_{t\to \pm \infty}\frac{2}{\pi}\cdot \frac{1}{t^{2}+1}=0.
\end{eqnarray*}
This yields that  $\phi$ satisfies assumption {\bf (H1)}.
Set $s:=\phi(t)$. Then $h(s)=\tan (\frac{\pi}{2}s)$.
We further have that the corresponding compactified system is given by
\begin{eqnarray}\label{app-sys-comp-1}
\begin{aligned}
\dot u &= v,\\
\dot v &=s^{2}(1-u^{2})v-u,\\
\dot s &=\frac{2}{\pi}\cos^{2}(\frac{\pi}{2}s)
\end{aligned}
\end{eqnarray}
for $(u,v,s)\in\mathbb{R}\times \mathbb{R}\times (-1,1)$, and
\begin{eqnarray}\label{app-sys-comp-2}
\begin{aligned}
\dot u &= v,\\
\dot v &=(1-u^{2})v-u,\\
\dot s &=0
\end{aligned}
\end{eqnarray}
for $(u,v,s)\in\mathbb{R}\times \mathbb{R}\times \{\pm 1\}$.
Note that
\begin{eqnarray*}
\lim_{s\to \pm 1^{\mp}}\left(-\frac{h''(s)}{(h'(s))^{2}}\right)
\!\!\!&=&\!\!\!
\lim_{t\to \pm \infty} \frac{\ddot{\phi}(t)}{\dot{\phi}(t)}
=
\lim_{t\to \pm \infty}\left(-\frac{2t}{t^{2}+1}\right)=0,\\
\lim_{s\to \pm 1^{\mp}}\frac{d}{ds}\mu(h(s))
\!\!\!&=&\!\!\!
\lim_{t\to \pm \infty}\frac{\dot \mu(t)}{\dot \phi(t)}
=\lim_{t\to \pm \infty}\frac{\dot \phi(t)}{\dot \phi(t)}
=1,\\
\lim_{s\to \pm 1^{\mp}}\left(\frac{2(h''(s))^{2}}{(h'(s))^{3}}-\frac{h^{(3)}(s)}{(h'(s))^{2}}\right)
\!\!\!&=&\!\!\!
\lim_{t\to \pm \infty} \frac{\phi^{(3)}(t)\dot{\phi}(t)-(\ddot\phi(t))^{2}}{(\dot{\phi}(t))^{3}}
=\lim_{t\to \pm \infty} \frac{\pi}{2}\cdot\frac{2(t^{2}-1)}{t^{2}+1}=\pi,\\
\lim_{s\to \pm 1^{\mp}}\frac{d^{2}}{ds^{2}}\mu(h(s))
\!\!\!&=&\!\!\!
\lim_{t\to \pm \infty}\left(\frac{\ddot \mu(t)\dot \phi(t)-\dot\mu(t)\ddot \phi(t)}{(\dot \phi(t))^{3}}\right)
=0.
\end{eqnarray*}
Then by Corollary \ref{coro-1}, the compactified system (see \eqref{app-sys-comp-1} and \eqref{app-sys-comp-2})
is at least $C^{2}$-smooth on the extended domain $\mathbb{R}^{2}\times [-1,1]$.
It is worth noting that higher regularity can be further verified by Theorem \ref{thm-transf-cond}.

We now use the results in section \ref{sec:IM} to study the dynamical behavior of system \eqref{app-sys-1}.
It follows from \cite[Theorem 1.6, p.60]{Hale-80} that the limit system \eqref{app-sys-limit} has exactly
two invariant sets, one is a unstable focus $(0,0)$ and the other one is an asymptotically stable periodic solution which is denoted by $\Lambda_{0}$.
See Figure \ref{fg-1}.
\begin{figure}[!htbp]
\centering
\subfigure[]{
\begin{minipage}[t]{0.45\linewidth}
\centering
\includegraphics[width=2.8in]{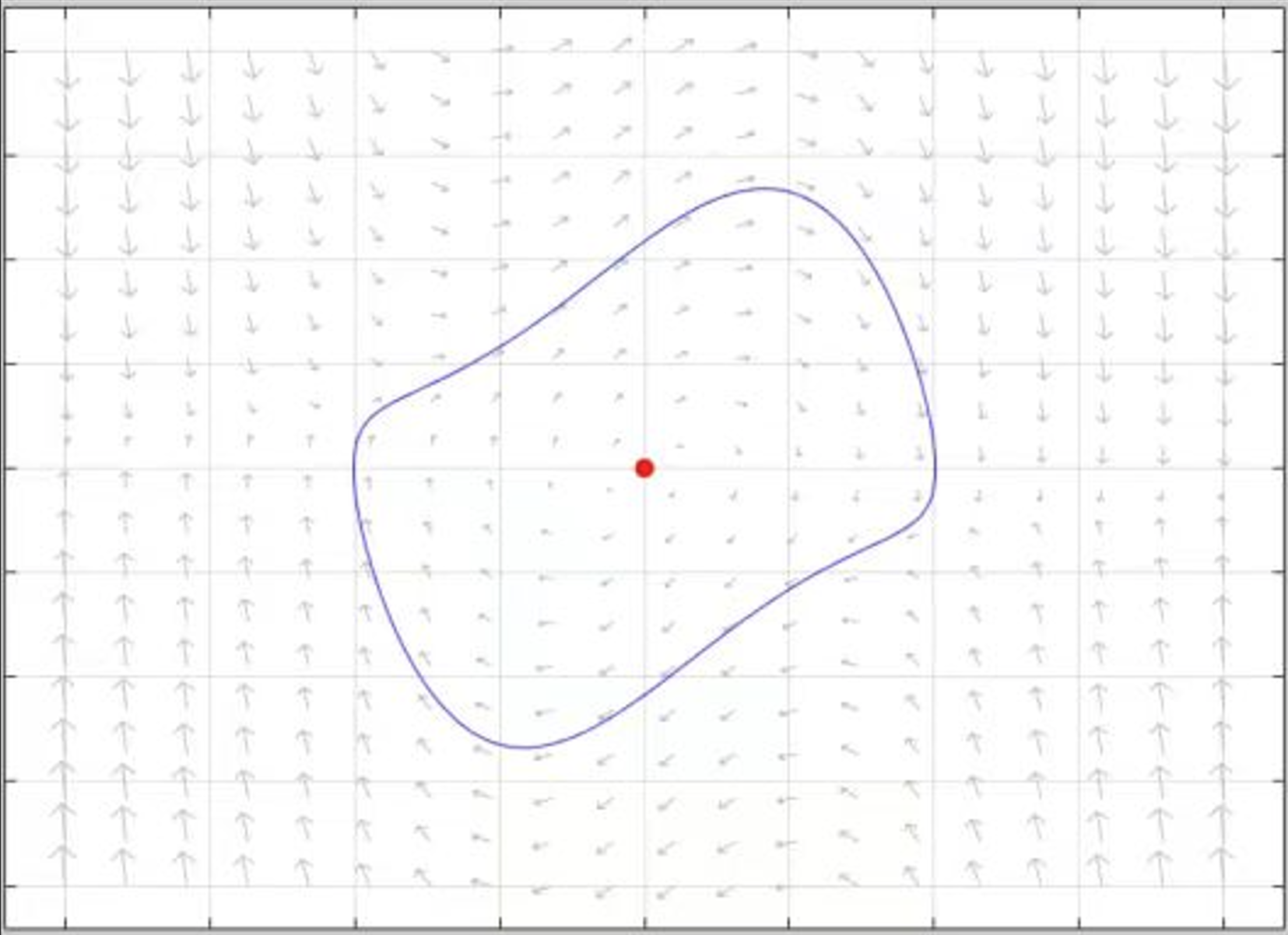}
\end{minipage}
\label{fg-1}
}%
\subfigure[]{
\begin{minipage}[t]{0.45\linewidth}
\centering
\includegraphics[width=2.8in]{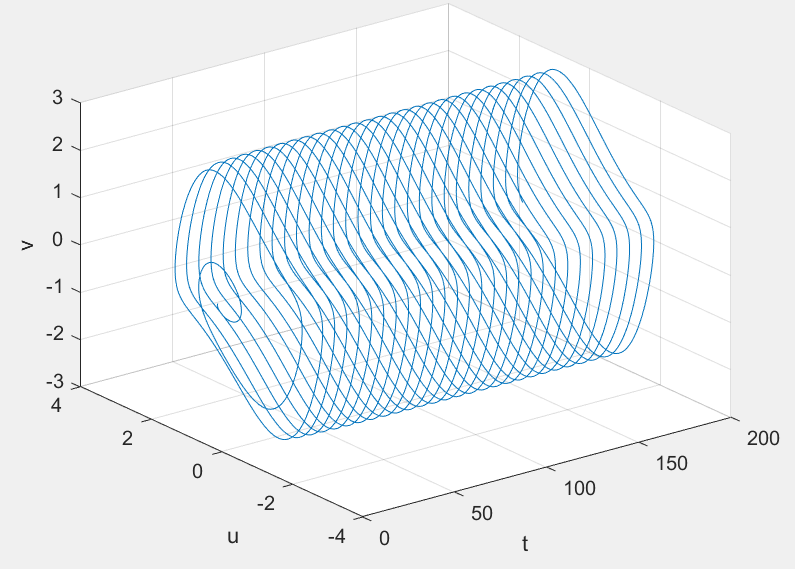}
\end{minipage}
\label{fg-2}
}%
\centering
\caption{
\ref{fg-1} Phase portrait of the limit system \eqref{app-sys-limit}.
  The red point is a unstable focus $(0,0)$ and the circle indicates an asymptotically stable periodic solution $\Lambda_{0}$.
\ref{fg-2} Solutions of system \eqref{app-sys-1} except $(0,0)$ are asymptotic to the periodic solution $\Lambda_{0}$ as $t\to +\infty$.
}
\end{figure}

Although we have
\[\lim_{t\to \pm \infty} \frac{\ddot{\phi}(t)}{\dot{\phi}(t)}
=
\lim_{t\to \pm \infty}\left(-\frac{2t}{t^{2}+1}\right)=0,
\]
which actually  implies that each invariant set of the compactified system gains one additional center direction,
the results in Theorem \ref{thm-NHIM-1} are applicable and we get the existence and uniqueness of center-stable manifolds for
the hyperbolic periodic solution $\Lambda_{0}$ of the future limit system,
and
$W^{\rm cs}_{\rm loc}(\tilde{\Lambda}^{+})=w_{\rm loc}^{\rm s}(\tilde{\Lambda}^{+})$ in the extended phase space of the compactified system,
where $\tilde{\Lambda}^{+}=(\Lambda_{0},1)$.
Recall that the stable manifold $W^{\rm s}(\tilde{\Lambda}^{+})$ of $\Lambda_{0}$ in the future limit system is $\mathbb{R}^{2}/\{(0,0)\}$.
This together with Theorem \ref{thm-NHIM-2} yields that
all solutions of system \eqref{app-sys-1} except $(0,0)$ are asymptotic to the periodic solution $\Lambda_{0}$
as $t\to +\infty$. See Figure \ref{fg-2}.

Actually, as shown in this example, the theory developed in this paper can be applied to investigate
the long term behavior of nonautonomous dynamical systems.
We hope that it can be also used to study forced waves \cite{Berestycki-Fang-18} and pullback attractors \cite{Kloeden-Rasmussen-11}.

\section*{Appendix A. The Fa\`a di Bruno formula}
\label{Append-A}

In this appendix,
we introduce the Fa\`a di Bruno formula for computing the partial derivatives of a function composition.
We also refer to \cite{Constantine-Savits-96,Faa-55} for more information on this formula.

Let $\varphi,\psi:\mathbb{R}\to \mathbb{R}$ be two $C^{n}$-smooth functions.
Suppose that $s_{0}=\psi(t_{0})$.
Then the $n$-th derivative of  $\Phi(t):=\varphi(\psi(t))$ at $t=t_{0}$ is given by
\renewcommand\theequation{A.1}
\begin{eqnarray}\label{FdB-formula}
\frac{d^{n}}{d t^{n}}\Phi(t_{0})
=\sum_{l=1}^{n}\frac{d^{l}}{ds^{l}}\varphi(s_{0})
\sum_{q(n,l)}n!\prod_{i=1}^{n}\frac{(\psi^{(i)}(t_{0}))^{\lambda_{i}}}{(\lambda_{i}!)(i!)^{\lambda_{i}}},
\end{eqnarray}
where
\[
\psi^{(i)}(t_{0})=\frac{d^{i}}{dt^{i}}\psi(t_{0}),
\ \ \ \ \
q(n,l)=\left\{(\lambda_{1},...,\lambda_{n}): \lambda_{i}\in \mathcal{N}_{0},\
       \sum_{i=1}^{n}\lambda_{i}=l,\ \sum_{i=1}^{n}i\lambda_{i}=n \right\}.
\]

\section*{Appendix B. Proof of Lemma \ref{lm:Phi-m-prpty}}
\label{Append-B}

We first prove the limit $t\to +\infty$. Noting that
\[
\Phi_{m}(t)\ln^{m}(t)=-1
\]
for sufficiently large $t>0$, we can compute
\renewcommand\theequation{B.1}
\begin{eqnarray}\label{eq:diff-1-0}
\dot\Phi_{m}(t)(t\prod_{i=1}^{m}\ln^{i}t)+\Phi_{m}(t)=0.
\end{eqnarray}
Differentiating \eqref{eq:diff-1-0} with respect to $t$ yields
\begin{eqnarray*}
\ddot\Phi_{m}(t)(t\prod_{i=1}^{m}\ln^{i}t)+\dot\Phi_{m}(t)(1+\frac{d}{dt}(t\prod_{i=1}^{m}\ln^{i}t))=0.
\end{eqnarray*}
Then we have
\renewcommand\theequation{B.2}
\begin{eqnarray}\label{limit-2-1}
\frac{\ddot\Phi_{m}(t)}{\dot\Phi_{m}(t)}
  =-(2+\sum_{j=1}^{m}\prod_{i\geq j}^{m}\ln^{i}t)(t\prod_{i=1}^{m}\ln^{i}t)^{-1}.
\end{eqnarray}
This implies that \eqref{limts:frac} holds for $n=1$.
Furthermore, by \eqref{eq:diff-1-0} we have
\renewcommand\theequation{B.3}
\begin{eqnarray}\label{eq:devit-1}
\Phi^{(n+1)}_{m}(t)(t\prod_{i=1}^{m}\ln^{i}t)
+\sum_{i=0}^{n-1}C_{n}^{i}\Phi^{(i+1)}_{m}(t)(\frac{d^{n-i}}{dt^{n-i}}(t\prod_{i=1}^{m}\ln^{i}t))
+\Phi^{(n)}_{m}(t)=0,
\ \ \ \ \ n\geq 1.
\end{eqnarray}
Dividing \eqref{eq:devit-1} by $(\dot\Phi_{m}(t))^{n}(t\prod_{i=1}^{m}\ln^{i}t)$ yields
\begin{eqnarray*}
\frac{\Phi^{(n+1)}_{m}(t)}{(\dot\Phi_{m}(t))^{n}}
\!\!\!&=&\!\!\!-\sum_{i=0}^{n-1}C_{n}^{i}\frac{\Phi^{(i+1)}_{m}(t)}{(\dot\Phi_{m}(t))^{i}}
  \cdot(\frac{d^{n-i}}{dt^{n-i}}(t\prod_{i=1}^{m}\ln^{i}t))\cdot\frac{\ln^{m}t}{(\dot\Phi_{m}(t))^{n-1-i}}
  -\frac{\Phi^{(n)}_{m}(t)\ln^{m}t}{(\dot\Phi_{m}(t))^{n-1}},
\ \ \ \ \ n\geq 1.
\end{eqnarray*}
Then for each $n\geq 2$, there exists a constant $t_{n}>0$ such that
\renewcommand\theequation{B.4}
\begin{eqnarray}\label{limit-n+1-n}
|\frac{\Phi^{(n+1)}_{m}(t)}{(\dot\Phi_{m}(t))^{n}}|
\leq t^{-1/3}+t^{1/3^{n}}\sum_{i=1}^{n-1}|\frac{\Phi^{(i+1)}_{m}(t)}{(\dot\Phi_{m}(t))^{i}}|,
\ \ \ \ \ t>t_{n}.
\end{eqnarray}
When $n=2$, we have
\begin{eqnarray*}
|\frac{\Phi^{(3)}_{m}(t)}{(\dot\Phi_{m}(t))^{2}}|\leq t^{-1/3}+t^{1/3^{2}}|\frac{\Phi^{(2)}_{m}(t)}{\dot\Phi_{m}(t)}|, \ \ \ \ \ t>t_{2}.
\end{eqnarray*}
This together with \eqref{limit-2-1} yields that \eqref{limts:frac} holds for $n=2$.
Note that
\[
\sum_{i=2}^{+\infty}\frac{1}{3^{i}}=\frac{1}{6}.
\]
Then by induction, there exists a constant $M_{k}$, dependent on $k$, such that
\[
|\frac{\Phi^{(n+1)}_{m}(t)}{(\dot\Phi_{m}(t))^{n}}|
\leq M_{k} (t^{-1/6}+t^{1/6}|\frac{\Phi^{(2)}_{m}(t)}{\dot\Phi_{m}(t)}|),
\ \ \ \ \
2\leq n\leq k.
\]
This together with \eqref{limit-2-1} yields that \eqref{limts:frac} holds for each $n$ with $2\leq n\leq k$.
The limit $t\to -\infty$ can be proved similarly.
Therefore, the proof is now complete.

\section*{Data availability}
No data was used for the research described in the article.

\section*{Acknowledgments}

This work is partially supported by the NSFC grant (No. 12101253),
the Scientific Research Foundation of CCNU (No. 31101222044) and
the CSC Visiting Scholarship.
The first author also thank Dr. Binyu Lu from Illinois Institute of Technology for his help in numerical simulations.

%\section*{Acknowledgments}

\bibliographystyle{amsplain}

\end{document}